\documentclass[11 pt]{article}
\usepackage{comment}
\usepackage{amsmath,amscd}
\usepackage{amstext}
\usepackage{amssymb,epsfig}
\usepackage{amsthm}
\usepackage{amsfonts}
\usepackage{fancybox, floatrow}
\usepackage{graphicx}
\usepackage[all]{xy}
\usepackage{subeqnarray}
\usepackage{enumitem}
\usepackage{cancel}
\usepackage{dsfont}
\usepackage{braket}
\usepackage[pagebackref, colorlinks = true,
linkcolor = red,
urlcolor  = blue,
citecolor = blue,
anchorcolor = blue]{hyperref}
\hypersetup{pdftitle={\@title},pdfauthor={\@author}}

\usepackage[dvipsnames]{xcolor}

\usepackage[utf8]{inputenc}      
\usepackage{tikz}

\newtheorem{mainth}{Main theorem}
\newtheorem{proposition}{Proposition}[section]

\newtheorem{lemma}[proposition]{Lemma}

\newtheorem{definition}[proposition]{Definition}

\def\ie{{\em i.e.,\ }}
\def\eg{{\em e.g.\ }}
\def\eps{\varepsilon}
\def\N{{\mathbb N}}
\def\Z{{\mathbb Z}}

\def\R{{\mathbb R}}

\newcommand {\CL}{{\mathcal L}}

\newcommand {\CP}{{\mathcal P}}

\def\s{\sigma}
\def\l{\lambda}

\def\1{ {\hbox{{\it 1}} \!\! I} }

\def\al{\alpha}
\def\be{\beta}
\def\de{\delta}
\def\ga{\gamma}
\def\om{\omega}
\def\Om{\Omega}

\def\8{\infty}

\def\disp{\displaystyle}

\renewcommand{\S}{\Sigma}
\usepackage{float,color}
\newcommand{\wt}{\widetilde}
\newcommand{\wh}{\widehat}
\newcommand{\ul}{\underline}
\newcommand{\BBone}{{1\!\!1}}

\newcommand{\e}{\mathbf{ e}}
\def\ninf{{n\to+\8}}
\def\minf{{m\to+\8}}

\def\epz{{\eps\to0}}

\makeatletter
\newcommand*\bdot{\mathpalette\bdot@{.7}}
\newcommand*\bdot@[2]{\mathbin{\vcenter{\hbox{\scalebox{#2}{$\m@th#1\bullet$}}}}}
\makeatother

\theoremstyle{definition}
\newtheorem{remark}{Remark}

\renewcommand{\L}{\Lambda}

\allowdisplaybreaks
\begin{document}
\title{Limits of equilibrium states for coupled weakly interacting systems. Application to the measure of maximal entropy}
\author{ Renaud Leplaideur\thanks{R. Leplaideur thanks Fasic 18 PaciDyn, CNRS-LIA FuMa , UQ and CNRS-IEA LOTO for financial supports.} 
    \footnote{Université de la Nouvelle-Calédonie 145, Avenue James Cook - BP R4 98 851 -
Nouméa Cedex Nouvelle Calédonie
    \texttt{renaud.leplaideur@unc.nc}}
}
\date{}
\maketitle

\begin{abstract}
We study metastability for symbolic dynamic. We prove that for a global system given by two independent sub-systems linked by a hole, and for a Lipschitz continuous potential, the global equilibrium state converges, as the hole shrinks, to a convex combination of the two independent equilibria in each component. Two kinds of convergence occur, depending on the assumptions on how long an orbit has to stay in each well. 

As a by-product, we show that this can be applied to a geometrical system inspired from \cite{tokman1} and for the measure of maximal entropy.


\end{abstract}

\section{Introduction}
\subsection{Background}
Metastability is a common phenomenon in nature. It has been studied from many different viewpoints. In this paper we want to focus on the dynamical systems and the statistical mechanics viewpoints. 
These two viewpoints are close, with, however, some differences. 

\medskip
Rigorous results  in the framework of dynamical systems in statistical mechanics can be found in \cite{Ruelle-rig}, and first  developments of  rigorous theories of metastability in the
framework of thermodynamics and statistical mechanics of Gibbsian ensembles can be found in \cite{Penrose-lebowitz}. There, metastability is presented as a system starting in  some phase that is likely to take a long time to escape, but  once out of this metastable phase, the system is extremely unlikely to return. 

\medskip
Along the years, with advances in the theory, this assumption on irreversibility of escapes  has been relaxed. 
 Nowadays, a typical metastable dynamical system consists of two adjacent systems linked via a small hole. 
The infrequent transitions between \emph{stable states}  are thus consequences of the small probability to enter into the hole. 
Therefore,  typical orbits will alternate long stays in one of the systems, hopping suddenly to the other one.

In that direction, most of studies focus on the \emph{physical-measure}, as the SRB measure or the Absolutely Continuous Invariant Measure (with respect to the Lebesgue measure) in dimension 1, see \eg \cite{bahsoun-saussol, bahsoun-vaienti, dolgopyat, dolgopyat-wright, kloeckner, ruelle-diff}.

Echoing these results, and within the statistical mechanics viewpoint,  the notion of \emph{quasi-stationary distribution} is introduced (see \eg  \cite{lelievre1, lelievre2})
Metastable systems is a process entering in wells. If this well is indeed a metastable region for the dynamics, then the process is led by this quasi-stationary distribution, 
which is  absolutely continuous with respect to the Lebesgue measure in the well.

\bigskip 
 On the other hand, and still within the statistical mechanics viewpoint, there are studies in  what is equivalent to symbolic spaces within the dynamical viewpoint. In that case, physical measures or Lebesgue measure have no meaning. We mention for instance \cite{Mahieu-Picco} and more recently \cite{Bovier-DenHollander, dubbeldam}. They deal with the mean-fields assumption for Curie-Weiss-Potts models. 
 
 Mean field systems can be interpreted within  the  \emph{ non-linear thermodynamical formalism} in the framework of dynamical systems (see \cite{leplaideur-watbled1, leplaideur-watbled2, BKL}).  Instead of maximizing the usual free energy
 $$P(\be):=\max\left\{h_{\mu}(T)+\be.\int \phi\,d\mu\right\},$$
 we maximize the quadratic quantity $\disp P_{2}(\be):= \max_{\mu} \left\{h_{\mu}(T)+\be.\left(\int \phi\,d\mu\right)^{2}\right\}$. Then, it turns out that any maximizing measure for $P_{2}(\be)$ is also a maximizing measure for the \emph{linear} $P(\wt\be)$, with $\wt \be=\wt\be(\be)$. 
 
 Consequently, it makes sense to study metastability for symbolic dynamics and more general Gibbs measures.
 

 \bigskip
 In this paper, the system we study  is given by two independent subshifts of finite type having the same pressure for some fixed potential. These subshifs are linked by some hole. To represent the size of the coupling hole that goes to 0, we remove the trajectories that do not stay sufficiently long time in each of the systems. 
 
 Under two different assumptions on how long the orbits have to stay in one of the system before hopping to the other one, Theorems \ref{th-main1-2} and \ref{th-main2-1} show that  the unique equilibrium state for the global potential $\phi$ converges to the  a fixed combination of the two independent equilibria in each independent system. Surprisingly, this does not depend on the choice of the potential.

 \bigskip
 Going back to more geometrical results, 
 Gonz\'alez-Tokman, Hunt and Wright studied in \cite{tokman1} the case of two separated piecewise expanding maps on the interval coupled by a small hole for $\eps>0$. For each $\eps>0$, there exists a unique absolutely continuous measure whose accumulation points for $\epz$ must be a convex combination of the two absolutely invariant measures  of the separated systems.  It has been proved that the measure converges and the limit depends on the ratio of the sizes of the holes (going to one part of the system to the other one).

For these geometrical  dynamical systems, two measures are usually considered as more relevant. The absolutely continuous one and the one with maximal entropy. Existence and uniqueness of measure of maximal entropy for piecewise expanding maps on the interval have been well studied, and we refer to \cite{buzzi-course} for a survey on that topic. 

It is thus a natural question to enquire the same kind of questions as in \cite{tokman1} but for the measure of maximal entropy. 

\bigskip
We emphasize that this question is technically much harder than the linear response for absolutely continuous invariant measures. Indeed, for this later case, the conformal measure is naturally given because it is the Lebesgue measure. Therefore, the main difficulty consists in studying the family of the densities as the parameter changes. For the piecewise expanding case, arguments like the one given in \cite{keller-liverani} can be used to control how the density changes. 

For the maximal entropy case, there is no more natural reference measure. Therefore, the study needs to work in both directions: to control the conformal measure and the associated density. This plus the fact that the systems (thus the combinatorics on orbits) changes with the parameter make the study difficult.

The difficulty is overcome by showing that  one geometrical dynamical system as in \cite{tokman1} is conjugated to a symbolic one where we can apply  Theorem \ref{th-main1-2}. In that case, this yields  that the measure of maximal entropy goes to $(\frac12,\frac12)$ combination of the two independent measures of maximal entropy (see Theorem \ref{th-main1-3}). 

%
%
%
%
%

\subsection{Settings for symbolic dynamics}\label{settingLargeMatrices}


\subsubsection{Settings at initial stage}
Let $A=(a_{i,j})$ and $D=(d_{i,j})$ be irreducible $K\times K$ and $N\times N$ matrices with entries in $\{0,1\}$.


We consider a transition matrix $M_{0}$, which is a $(K+N)\times (K+N)$ whose restriction to the both diagonal blocks $K\times K$ and $N\times N$ are respectively $A$ and $D$. We also assume that $M_{0}$ is irreducible: there exists some $n$ such that $M_{0}^{n}$ has only positive entries. Let $\S_{0}$ be the subshift of finite type with transition matrix $M_{0}$. 

A point $\om\in \S_{0}$ is an infinite sequence/word $\om=\om_{0}\om_{1}\ldots$ of digits. Digits may be either in $\{\al_{1},\ldots , \al_{K}\}$ or  in $\{\de_{1},\ldots \de_{N}\}$, the whole alphabet being $\{\al_{1},\ldots, \al_{K},\de_{1},\ldots, \de_{N}\}$. Digits in $\{\al_{1},\ldots , \al_{K}\}$  are said to be the $\al$-digits, digits in $\{\de_{1},\ldots \de_{N}\}$ are said to be the $\de$-digits. 
Each infinite word $\om$ is thus an alternation of strings of $\al$-digits and $\de$-digits. It shall be written as $\al^{m_{0}}\de^{m_{1}}\al^{m_{2}}\ldots$.
The surviving sets will be $\S_{A}$ and $\S_{D}$, that are sets of infinite words only with $\al$-digits  or only with $\de$-digits. 

We consider the usual distance on $\S_{0}$ defined by 
$$d(\om,\om')=2^{-n(\om,\om')},$$
where $n(\om,\om')=\min\{k\in\N,\ \om_{k}\neq \om'_{k}\}$ (this may be $+\8$ iff $\om=\om'$). 

A $n$-cylinder $[\om_{0}\ldots\om_{n-1}]$ is the set of $x=x_{0}x_{1}\ldots$ such that $x_{i}=\om_{i}$ for $i\le n-1$. The union of the 1-cylinders of the $\al$-digits form the neighborhood $\Om_{A}:=B(\S_{A},1)$. The union of the $1$-cylinders of the $\de$-digit form  the neighborhood $\Om_{D}:=B(\S_{D},1)$.

\bigskip
We consider a Lipschitz continuous potential $\phi:\S_{0}\to\R$ and we assume that both  $\S_{A}$  and $\S_{D}$ have the same pressure associated to $\phi$, that is denoted by $P$. Associated equilibrium states are respectively denoted by $\mu_{A}$ and $\mu_{D}$.

\medskip

\paragraph{The coupling hole.}

For simplicity, the $\al_{j}$'s  (with subscript $j$) are the $\al$-digits that may be followed by a $\de$-digits. These letters are the $\de_{l}$'s (with subscript $l$). Similarly, the $\de_{i}$'s are the $\de$-digits that may be followed by at least one $\al$-digit, say the $\al_{k}$'s. 

The cylinders $[\al_{j}\de_{l}]$ and the cylinders $[\de_{i}\al_{k}]$  are called {\bf green cylinders}. 
The union of the green cylinders is called the hole, and has one part in $\Om_{A}$ (trajectories going from $\Om_{A}$ to $\Om_{D}$, \ie cylinders $[\al_{j}\de_{l}]$ ) and one part in $\Om_{D}$ (\ie cylinders $[\de_{i}\al_{k}]$). 

\subsubsection{First reduction of the coupling hole and Main Theorem \ref{th-main1-2}}

The reduction of the hole is done as follows. 
\begin{enumerate}

\item One considers $A'$ a $K\times K$ matrix and $D'$ a $N\times N$ matrix with entries in $\{0,1\}$ such that if one entry in $A$ or $D$ is zero, then the corresponding entry in $A'$ or $D'$ is also equal to zero. 
\item The matrix $A'$ (resp.$D'$) has strictly less transitions than $A$ (resp. $D$).
\item The lines correponding to the $\al_{k}$'s in $A'$ and the $\de_{l}$'s in $D'$ have some entries equal to 1. 
\end{enumerate}
This means that there are less transitions  with $A'$ than with $A$, and with $D'$ than with $D$, but there exists transitions $\al_{k}\al\ldots$ and $\de_{l}\de\ldots$ for all the $\al_{k}$'s and the $\de_{l}$'s. Conditions 1. and 2. can be summarized by the  formulas 

$$A'<A \text{ and }D'<D.$$

Then, one fixes two increasing sequence of integers $(n_{m})$  and $(n'_{m})$ both going to $+\8$ as $m\to+\8$, and one considers the new subshift of finite type $\S_{m}\subset \S_{0}$ with pressure $P_{m}$ such as 
 the forbidden words (in addition to those in $\S_{0}$)  are
 \begin{itemize}
\item $\al_{j}\de_{l}\de^{n}\de_{i}\al_{k}$ with $n<n'_{m}-2$,
\item $\al_{j}\de_{l}\de^{n}\de_{i}\al_{k}$ with $n\ge n'_{m}-2$ but $\de_{l}\de^{n'_{m}-2}$ is not $D'$-eligible, 
\item $\de_{i}\al_{k}\al ^{n}\al_{j}\de_{l}$ with $n<n_{m}-2$,
\item $\de_{i}\al_{k}\al ^{n}\al_{j}\de_{l}$ with $n\ge n_{m}-2$ but $\al_{k}\al ^{n_{m}-2}$ is not $A'$-eligible. 
\end{itemize}

We point out\footnote{and left it to the reader to check.} that the global system $\S_{m}$ is irreducible since $\S_{A}$ and $\S_{D}$ are irreducible. 


\medskip
The hole at step $m$ is the union of cylinders 
of the form $[\al_{j}\de_{l}\de^{n'_{m}-2}]$ and $[\de_{i}\al_{k}\al^{n_{m}-2}]$, with the restriction that $\de_{l}\de^{n'_{m}-2}$ is $D'$-eligible and $\al_{k}\al^{n_{m}-2}$ is $A'$-eligible.

\bigskip
As $\S_{m}$ is irreducible, it admits a unique equilibrium state for $\phi$, denoted  $\wh\mu_{m}$. Because $A'$ and $D'$ have less allowed transitions than $A$ and $D$, subshifts $\S_{m}$ form a decreasing sequence of subshifts, converging to $\S_{A}\sqcup \S_{D}$. 
Uniqueness of the equilibrium state for irreducible subshifts  yields that $P_{m}$ decreases to $P$.

Then our first result is: 
\begin{mainth}
\label{th-main1-2}
Let $\wh\mu_{m}$ be the unique equilibrium state for $\S_{m}$ associated to $\phi$.
Assume that $\disp \lim_{\minf}\dfrac{n_{m}}{n'_{m}}=\theta\in]0,+\8[$ holds as $m$ goes to $+\8$. 
Then $\wh\mu_{m}$ converges to $\dfrac1{2}(\mu_{A}+\mu_{D})$
 \end{mainth}

\subsubsection{Second reduction of the coupling hole and Main Theorem \ref{th-main2-1}}
The second way to cut the coupling is actually more natural for the symbolic dynamics. This is why we state it here now, in view to develop this work later. 

The reduction is obtained by picking $A'=A$ and $D'=D$.

To be sure that the system is still irreducible we require that $n_{1}$ and $n'_{1}$ are sufficiently  big such that $A^{n_{1}}$ and $D^{n'_{1}}$ have only positive entries. Then, $\S_{m}$ admits a unique equilibrium state for $\phi$. We denote it by $\wh\mu_{m}$.
Our second result is: 
\begin{mainth}
\label{th-main2-1}
Let $\wh\mu_{m}$ be the unique equilibrium state for $\S_{m}$ associated to $\phi$. 
Let $\theta$ be in $[0,+\8]$. 
If $\disp \lim_{\minf}\dfrac{n_{m}}{n'_{m}}=\theta$, then,  $\wh\mu_{m}$ converges to $\dfrac1{1+\theta}(\theta.\mu_{A}+\mu_{D})$ as $m$ goes to $+\8$. 
\end{mainth}

\begin{remark}
\label{rem-theteth}
We emphasize that in Theorem \ref{th-main2-1} we may have $\theta=0$ or $\theta=+\8$. Obviously, in that later case, $\dfrac{\theta}{1+\theta}=1$ and $\dfrac1{1+\theta}=0$.  
$\blacksquare$\end{remark}

\subsection{Settings for geometrical dynamical systems}
The family of maps we consider is given by Figure \ref{fig-map1}. They are assumed to be \emph{almost} piecewise expanding. 

Piecewise expanding means that there exists $\l>1$ such that on each branch $|T_{\eps_{0}}(x)-T_{\eps_{0}}(y)|\ge \l|x-y|$.
In our construction, such an inequality holds everywhere except on the top of the cusp that is the part which shrinks to a single point. However, the maps are non-uniformly expanding, in the sense that they are conjugated to a map with positive Lyapunov exponent.

\begin{figure}[htbp]
\begin{center}
\includegraphics[scale=0.8]{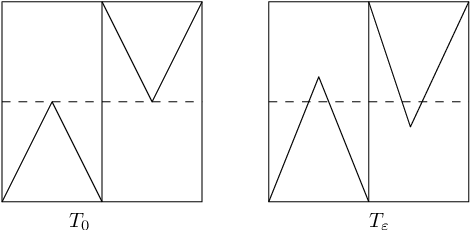}
\caption{{\bf The family of maps $T_{\eps}$}}
\label{fig-map1}
\end{center}
\end{figure}

%
In the following the Left-subsystem is the map $T_{0}$ restricted to $\left[0,\frac12\right]$ and the Right-subsystem is the map $T_{0}$ restricted to $\left[\frac12,1\right]$. Each admits an unique measure of maximal entropy, respectively denoted by $\mu_{top,L}$ and $\mu_{top, R}$.

We also emphasize that we do not need the maps to be symmetric (despite that for commodity,  this appears on the picture). 

Then we get 
\begin{mainth}\label{th-main1-3}
There is way to construct a family of maps $T_{\eps}$  with $\eps\to 0$ such that 
the measure of maximal entropy $\mu_{top,\eps}$ for $\eps>0$ goes to $\frac12(\mu_{top,L}+\mu_{top,R})$. 
\end{mainth}

\subsection{Open questions, Plan of the paper and Aknowledgment}

\subsubsection{Open questions}
For the geometrical dynamical system, it would be interesting to get the convergence without using the symbolic dynamics. However as it is said above, the difficulty is to construct for positive $\eps$ the conformal measure for the transfer operator. We are aware of  attempts in that direction, but as far as we know, without success. 

Theorem \ref{th-main1-3} states that there is a way to get the convergence. It is thus limited in the choices of the $\eps$ and for the branches. It seems clear we cannot hope convergence for any maps $T_{\eps}$ with the good shape. We can for instance introduce big distorsion between the left and the right sides. Such a distorsion could annihilate the assumption $\theta\neq 0,+\8$ in Theorem \ref{th-main1-2}, and then affect the convergence. 
It would be nice to inquire how big is the set of maps for which one can say something. 

Theorem \ref{th-main2-1} gives a nice control on the ratio of the measures ``on each side''. It would be interesting to develop this study for more than 2 independent systems and a more complicated coupling between these systems. 

It is not impossible that our strategy to get Theorem \ref{th-main1-3} (via a conjugacy) could work for the ACIM measure. Indeed, this measure is the Gibbs measure associated to the potential $-\log T_{\eps}'$. It might be that we could control de conjugacy as the hole shrinks in term of $T_{\eps}'$. Then, we could apply an extension of  Theorem \ref{th-main1-2} but with a converging sequence of potentiels in the symbolic space.

\subsubsection{Plan of the paper}

The proofs for  Theorem \ref{th-main1-2} and Theorem \ref{th-main2-1} involve inducing schemes. We consider the first return map from the hole to the hole (at step $m$) and study an adapted transfer operator. 
For that purpose we do recalls in Section \ref{sec-recallandmore} on Thermodynamic formalism and inducing schemes. 

We first recall some classical properties on transfer operators for irreducible subshifts of finite type. 

In \ref{subsec-extension} we explain that even if we have a thin subsystem with empty interior in a bigger system, it is possible to consider  transfer operators for the small system but ``acting on the big system''. This was already mentioned in  \cite{leplaideur-flot} and we recall some facts about it. 

In Subsection \ref{sec-theremoalaleplaide} we recall a way to recover the \emph{global} equilibrium state from the thermodynamical formalism on an inducing scheme (here the first return map in the hole). This is a repetition  but adapted to our case of what can be found in  \cite{leplaideur-synthese}. 

Then, the proof of Theorem \ref{th-main1-2} is done in Section \ref{sec-proof-main-1-2}. It has three steps. In the first subsection we study the convergence for the conformal measure (for induced scheme). We introduce some quantities $\l'_{m}$ and $\l_{m}$ that play a role in the convergence. 
In the next Subsection we show that there is convergence for the eigen-function (for the induced scheme). This depends on a normalization that depends on/if one of the quantities $\l'_{m}$ and $\l_{m}$ goes to $+\8$ (or not). 
Then, proof of the convergence is done in \ref{subsec-computereturn} where we compute expectations of the return times. 

The proof of Theorem \ref{th-main2-1} is done in Section \ref{sec-proofmain2-1}. The main work is to explain how we can deduce it from Theorem \ref{th-main1-2} by doing $A'=A$ and $D'=D$. 

Finally, proof of Theorem \ref{th-main1-3} is done in  Section \ref{sec-proofmain3}. In the first subsection we construct  a geometric model conjugated to $\S_{0}$, and then get a sequence of maps $\wt T_{\eps_{m}}$ that are conjugated to $\S_{m}$. In the second subsection we show that these maps are conjugated to some maps $T_{\eps_{m}}$ with the good shape.

\subsubsection{Acknowledgment}
Part of this work has been done during visits in UQ. We thus warmly thank Cecilia Gonz\'alez-Tokman and Jason Atnip for nice discussions. 

\section{Recalls and technical results on thermodynamical formalism}
\label{sec-recallandmore}
\subsection{Equilibrium states and transfer operator}
We remind that for a given dynamical system $(X,T)$ and a potential $V:X\to \R$, the pressure  is defined by 
$$\CP:=\sup_{\mu\ T-inv}\left\{h_{\mu}(T)+\int V\,d\mu\right\},$$
where $h_{\mu}(T)$ is the Kolmogorov of the $T$-invariant probability $\mu$. Any measure realizing the supremum is called an equilibrium state. 

In the case where the dynamical systems is an irreducible subshift of finite type and $V$ is Lipschitz continuous then, the system admits a unique equilibrium state. Furthermore it is a Gibbs measure, and is obtained via the transfer operator as follows. 

The transfer operator is defined by 
$$L(g)(x):=\sum_{y,\ T(y)=x}e^{V(y)}g(y),$$
with $g:X\to \R$.  If $(X,T)$ is an irreducible subshift of finite type and $V$ is Lipschitz continuous, then $L$ acts on continuous function and also on Lipschitz-continuous functions. 
Furthermore it satisfies the spectral gap property: 

\begin{enumerate}
\item The spectral radius $e^{\l}$ is a simple dominating eigenvalue on the Lipschitz functions,
\item there exists a probability measure $\nu$, called the conformal measure such that for every continuous $g$, 
$$\int L(g)\,d\nu=e^{\l} \int g\,d\nu,$$
\item there exists a unique Lipschitz continuous eigenfunction $H$ (up to the normalization $ \disp\int H\,d\nu=1$) satisfying $L(H)=e^{\l}.H$. Furthermore $H$ is positive and is equal to 
\begin{equation}
\label{eq-1-hlimitecesaro}
\disp \lim_{n\to+\8}\frac1n\sum_{k=0}^{n-1}e^{-k\l}L^{k}(\BBone),
\end{equation}
where the convergence is for $||\ ||_{\8}$-norm. 
\item The function $H$ is bounded away from below and from above (bounds comes from mixing). 
\item The operator $L$ acts as 
\begin{equation}
\label{equ2-splitope}
\forall n\le 0,\quad L^{n}(g)=e^{n\l} \int g\,d\nu.H+e^{n(\l-\rho)}\Psi^{n}(g),
\end{equation}
where $\Psi$ is bounded with spectral radius lower than 1 on Lipschitz continuous functions. 
\item The unique equilibrium state $\mu$ is equal to $H\otimes \nu$, and $P=\l$. 
\end{enumerate}

\bigskip
We remind here what \emph{conformality} for $\nu$ means. 
If $[C]$ is a cylinder where $T$ is one-to-one, then 

\begin{equation}
\label{eq1-conform}
\nu([C])=e^{-\l}\int_{T([C])}e^{V(T^{-1}(.))}d\nu,
\end{equation}
where $T^{-1}$ is well-define since $T$ is one-to-one on $[C]$.

\subsection{Extension of the eigenfunctions for transfer operators}
\label{subsec-extension}
\subsubsection{A general result}

We consider two irreducible subshifts of finite type on the {\bf same finite alphabet}, respectively given by two transition matrices $E'$ and $E$. We assume that the matrices satisfies $E'<E$, in the sense that if a transition $i\to j$ is allowed in $E'$ then it is allowed in $E$. 

The respective subshifts are denoted by $\S_{E'}$ and $\S_{E}$. In the case when some transitions are allowed in $E$ but not in $E'$, then $\S_{E'}$ has empty-interior in $\S_{E}$. 

If we pick some potential $\varphi:\S_{E}\to\R$, we point-out here that it is possible to extend the transfer operator for $\S_{E'}$ and make it act on $\S_{E}$. This was already mentioned in \cite{leplaideur-flot}. 

\begin{proposition}
\label{prop-extension eigenfunc}
The transfer operator $\CL_{E'}$ associated to the potential $\varphi$ can be extended to an operator acting on continuous and on Lipschitz functions defined on $\S_{E}$. The spectral properties, including decompositions still hold for that extended operator. 
In particular, the eigenfunction $H_{E'}$ can be extended to an eigenfunction for the extended operator on $\S_{E}$. It is Lipschitz continuous. 
\end{proposition}
\begin{proof}
Since $\S_{E'}$ and $\S_{E}$ are defined on the same alphabet, any 1-cylinder for $E'$ is included into a 1-cylinder for $E$. The crucial point is that a transfer operator is defined by the inverse branches. 
For $x=x_{0}x_{1}\ldots$ in $\S_{E}$, the transfer operators $\CL^{n}_{E'}$ considers pre-images of the form 
$$y_{-n}y_{-n+1}\ldots y_{-1}x_{0}x_{1}\ldots,$$
with $y_{-n}y_{-n+1}\ldots y_{-1}x_{0}$ $E'$-eligible (hence $E$-eligible).

Therefore, the Lasota-Yorke inequality (which is the key point to get the spectral gap) holds for this extended transfer operator $\CL_{E'}$ defined on continuous and Lipschitz-continuous functions from $\S_{E}$ to $\R$. We can thus apply the Ionescu-Tulcea \& Marinescu theorem, which gives the spectral decomposition. 
\end{proof}

We however point-out that if the eigen-function $H_{E'}$ can be extended to an eigenfunction defined on $\S_{E}$, there is still a unique eigen-measure, that is $\nu_{E'}$. It has full-support in $\S_{E'}$ which has empty-interior in $\S_{E}$. By construction, the restriction of $H_{E}$ to $\S_{E'}$ is $H_{E'}$ and we still have 
$$\int H_{E}\,d\nu_{E'}=\int H_{E'}\,d\nu_{E'}=1.$$

\subsubsection{Application in our case}

%
%
%

In the following we shall use Prop. \ref{prop-extension eigenfunc} to several cases: 

\paragraph{Extension of $\S_{A'}$ in $\S_{A}$ and $\S_{D'}$ in $\S_{D}$.}
In that case we use Prop. \ref{prop-extension eigenfunc} with $E'=A'$ (resp. $E'=D'$) and $E=A$ (resp. $E=D$), if necessary restricted to 1-cylinders in common to $A'$ and $A$ (resp. $D'$ and $D$). 

We shall define by $\wt H_{A'}$ and $\wt H_{D'}$ the respective extended eigen-functions to $\S_{A}$ and $\S_{D}$.

The spectral decomposition of the operator are still given by 
\begin{equation}
\label{equ1-claprime-extend}
\forall n\ge 0,\ \forall g\ \text{Lipschitz } \e^{-nP_{A'}}\CL_{A'}^{n}(g)=\int g\,d\nu_{A'}\wt H_{A'}+e^{-n\rho}\Psi_{A'}^{n}(g),
\end{equation}
and 
\begin{equation}
\label{equ1-cldprime-extend}
\forall n\ge 0,\ \forall g\ \text{Lipschitz } \e^{-nP_{D'}}\CL_{D'}^{n}(g)=\int g\,d\nu_{D'}\wt H_{D'}+e^{-n\rho}\Psi_{D'}^{n}(g).
\end{equation}
 The spectral gap is  $\rho>0$. We can always assume it is the same for both operators. 

\paragraph{Extension of $\CL_{A}$ and $\CL_{D}$ respectively to $\Om_{A}$ and $\Om_{D}$}
In that case we use Prop. \ref{prop-extension eigenfunc} with $E'=A$ and $E=M_{0}$ but restricted to $\Om_{A}$, that is cylinders $[\al]$ for which inverse branches for $\CL_{A}$ make sense. 

Similarly we use it with $E'=D$ and $E=M_{0}$ restricted to $\Om_{D}$. 

 For simplicity we will keep the same notation for the extended operators except for eigenfunction that we shall denote by $\wt H_{A}$ and $\wt H_{D}$.

The spectral decomposition of the operator are still given by 
\begin{equation}
\label{equ1-cla-extend}
\forall n\ge 0,\ \forall g\ \text{Lipschitz } \e^{-nP}\CL_{A}^{n}(g)=\int g\,d\nu_{A}\wt H_{A}+e^{-n\rho}\Psi_{A}^{n}(g),
\end{equation}
and 
\begin{equation}
\label{equ1-cld-extend}
\forall n\ge 0,\ \forall g\ \text{Lipschitz } \e^{-nP}\CL_{D}^{n}(g)=\int g\,d\nu_{D}\wt H_{D}+e^{-n\rho}\Psi_{D}^{n}(g),
\end{equation}
since $e^{P}$ is the unique and simple dominating eigenvalue. The spectral gap is  $\rho>0$. We can always assume it is the same for both operators. 

\subsection{Inducing scheme, local and global equilibrium states}
\label{sec-theremoalaleplaide}

\subsubsection{What is local equilibrium and link with global equilibrium}

We recall here a method to recover/built the \emph{global} equilibrium state $\wh\mu_{m}$ from an inducing scheme and what we call a \emph{local} equilibrium state. Many of the results here can be found in \cite{leplaideur-synthese}. 

For simplicity we give the statements with the settings we shall use later. 

\bigskip
We consider the first return map, say $g_{m}$ into the union of the green cylinders, say $G$. The first return time is denoted by $r_{m}$. 
The letter $m$ reminds that the map depends of the size of the coupling hole. 

Now, we define the transfer operator as 
$$\CL_{G,m}(f)(x):=\sum_{y\in G,\ g_{m}(y)=x}e^{S_{r_{m}(y)}(\phi)(y)-r_{m}(y)P_{m}}f(y).$$

Because all the green-cylinders have the same length, 2, for any green-cylinder, all the points in that cylinder have the same inverse-branches for $g_{m}$. Then it turns out that $\CL_{G,m}$ acts on continuous and on Lipschitz functions defined on $G$. Furthermore, the classical theory can be adapted and the following holds:
\begin{enumerate}
\item $\CL_{G,m}$ has 1 for spectral radius. It is a simple and unique dominating eigenvalue. 
\item There is a conformal measure, say $\nu_{m}$, that is the eigen-measure for $\CL_{G,m}^{*}$ (when acting on continuous functions) associated to the eigenvalue 1. 
\item  There is a unique eigen-function, say $H_{m}$ for $\CL_{G,m}$  associated to the eigenvalue 1, uniqueness is up to the normalization 
$$\int H_{m}\,d\nu_{m}=1.$$
\item The function $H_{m}$ is equal to $\disp\lim_{n\to+\8}\frac1n\sum_{k=0}^{n-1}\CL_{G,m}^{k}(\BBone_{G})$, the limit is for the $||\ ||_{\8}$ norm. 
\item The measure $\mu_{m}:=H_{m}\otimes\nu_{m}$ is the unique equilibrium state for $(G,g_{m})$ and the potential $\Phi_{m}:=S_{r_{m}(*)}(\phi)(*)-r_{m}(*)P_{m}$. It is referred to the \emph{local} equilibrium state. 
\item The \emph{global} equilibrium state $\wh\mu_{m}$ satisfies $\dfrac{\wh\mu_{m}(\cdot\cap G)}{\wh\mu_{m}(G)}=\mu_{m}(\cdot)$, and 
$$\dfrac1{\wh\mu_{m}(G)}=\int_{G} r_{m}\,d\mu_{m}.$$
\end{enumerate}

The {\bf red cylinders} are eligible\footnote{This means that the words $\de_{l}\de^{n_{m}-2}$ are $D'$ eligible and the words $\al_{k}\al^{n_{m}-2}$ are $A'$-eligible} cylinders $[\al_{j}\de_{l}\de^{n_{m}-2}\ldots \de_{i}\al_{k}]$ or $[\de_{i}\al_{k}\al^{n_{m}-2}\ldots \al_{j}\de_{l}]$. Their image by the first return in the hole are the green cylinders. 
In the following we shall write $red \subset [\al_{j}\de_{l}]\to [\de_{i}\al_{k}]$ to say we consider a cylinder $[\al_{j}\de_{l}\de^{n_{m}-2}\ldots \de_{i}\al_{k}]$. Moreover, $[\de_{i}\al_{k}]$ will be called the ``target (green) cylinder'' for the red cylinder. 

Then, for any $x$ in say  $[\de_{i}\al_{k}]$, the $y$'s in $G$, such that $g_{m}(y)=x$ holds, are exactly the points $\al_{j}\de_{l}\de^{n_{m}-2}\ldots x$. In other words, any $red \subset [\al_{j}\de_{l}]\to [\de_{i}\al_{k}]$ contains exactly one pre-image of $x$ for $g_{m}$. Furthermore, all the points in that $red \subset [\al_{j}\de_{l}]\to [\de_{i}\al_{k}]$ have the same return time which is the length of the cylinder minus 1. By construction  the return time is at least $n_{m}$. 

Hence we have 

\begin{eqnarray}
\CL_{G,m}(f)(x)&=&\sum_{n=n_{m}}^{+\8}\sum_{y\in red \subset [\al_{j}\de_{l}]\to [\de_{i}\al_{k}],\ r(y)=n}e^{S_{n}(\phi)(y)-nP_{m}}f(y).\label{eq-lgmred}\\
&=& \sum_{n=n_{m}-2}^{+\8}\sum_{[\al_{j}\de_{l}]}\sum_{\de \text{ possible}, |\de|=n}e^{S_{n+2}(\phi)(\al_{j}\de_{l}\de x)-(n+2)P_{m}}f(\al_{j}\de_{l}\de x).\label{eq-lgmdel},
\end{eqnarray}
where $\de$ possible means $\de_{l}\de \de_{i}$ is eligible.  

\subsubsection{Some extra results }
\label{subsec-extraequiconti}
In view to prove that all the $H_{m}$ are bounded from above and equicontinuous, we need some extra results. 

\begin{lemma}
\label{lem-controlsamecyl}
There exists an universal constant $C$ such that for every $x$ and $x'$  in the same green-cylinder $[G]$ and for every $m$,
$$e^{-C}\le \dfrac{\CL_{G,m}^{k}(\BBone_{[G]}(x)}{\CL_{G,m}^{k}(\BBone_{[G]}(x')}\le e^{C}$$
\end{lemma}
\begin{proof}
For simplicity we assume that $x$ and $x'$ both belong to $[\de_{i}\al_{k}]$. Then, any $red \subset [\al_{j}\de_{l}]\to [\de_{i}\al_{k}]$ contains one pre-image  for $x$, say $y$ and one pre-image for $x'$ say $y'$. Both points have the same return time,  say $r$.

Because $\phi$ is Lipschitz continuous we get:
$$\left|S_{r(y)}(\phi)(y)-S_{r(y')}(\phi)(y')\right|\le \frac{C_{\phi}}2\dfrac{1-\left(\frac12\right)^{r(y)}}{1-\frac12}d(x,x')\le C_{\phi}d(x,x')\le \dfrac{C_{\phi}}{2}.$$
This holds for any couple of pre-images of $x$ and $x'$ in the same red-cylinder, and \eqref{eq-lgmred} yields the result with $C=\dfrac{C_{\phi}}2$ for $k=1$.

For $k\ge 2$, if $x$ and $x'$ are in the same green-cylinder  and if $y=y_{0}\ldots y_{n}x$ is a preimage of order $k$ for $g_{m}$ of $x$, then $y':=y_{0}\ldots y_{n}x'$ is also a preimage of $x'$ of order $k$ for $g_{m}$. Again, one can associate pre-images  by pair and we still get 
$$|S_{n}(\phi)(y)-S_{n}(\phi)(y)|\le C_{\phi}d(x,x')\le\frac{C_{\phi}}2,$$
which yields the result. 
\end{proof}

Because $H_{m}=\lim_{\ninf}\frac1n\sum_{k=0}^{n-1}\CL_{G,m}^{k}(\BBone)$, Lemma \ref{lem-controlsamecyl} immediately yields 
\begin{lemma}
\label{lem-distoboundedHm}
For $x$ and $x'$ in the same green-cylinder, 
$$e^{-C_{\phi}/2}\le\dfrac{H_{m}(x)}{H_{m}(x')}\le e^{C_{\phi}/2}.$$
\end{lemma}

We also get 
\begin{lemma}
\label{lem-controlhmmax}
$\sum_{[C]\ Green-cylinder}||H_{m|[C]}||_{\8}\nu_{m}([C])\le e^{C_{\phi}/2}$. 
\end{lemma}
\begin{proof}
If $[C]$ is a green-cylinder, then $H_{m|[C]}$ is continuous on the compact $[C]$ hence is bounded and reaches its bounds. 
Let $x_{[C]}$ be a point in $[C]$  where  $H_{m}$  is maximal. 
Then for any $x'\in [C]$, 
$$H_{m}(x')\ge e^{-C_{\phi}/2}H_{m}(x_{[C]})=e^{-C_{\phi}/2}||H_{m,|[C]}||_{\8}.$$
Then, 
$$1=\int H_{m}\,d\nu_{m}=\sum_{[C]\ Green-cylinder}\int_{[C]}H_{m}(x')d\nu_{m}(x')$$
yields 
$$1\ge \sum_{[C]\ Green-cylinder}e^{-C_{\phi}/2}||H_{m|[C]}||_{\8}\nu_{m}([C]).$$
\end{proof}

Concerning the control of the distortion we get a better result: 
\begin{proposition}
\label{prop-hmequilip}
The $H_{m}$ are all equicontinuous (as $m$ changes). 
\end{proposition}
\begin{proof}
The computation done in the proof of Lemma \ref{lem-controlsamecyl} actually gives a finer result : 
For $x$ and $x'$ in the same green-cylinder, for $k$ integer, 
\begin{equation}
\label{equ1disto}
e^{-\frac{C_{\phi}}2d(x,x')}\le \dfrac{\CL_{G,m}^{k}(\BBone_{G}(x)}{\CL_{G,m}^{k}(\BBone_{G}(x')}\le e^{\frac{C_{\phi}}2d(x,x')}.
\end{equation}
This yields 
\begin{equation}
\label{equ1disto}
e^{-\frac{C_{\phi}}2d(x,x')}\le \frac{H_{m}(x)}{H_{m}(x')}\le e^{\frac{C_{\phi}}2d(x,x')}.
\end{equation}
Hence we get  
\begin{eqnarray*}
H_{m}(x)-H_{m}(x')&=& H_{m}(x)(e^{\pm\frac{C_{\phi}}2d(x,x')}-1)\\
&=& (e^{\pm\frac{C_{\phi}}2d(x,x')}-1)e^{\pm \frac{C_{\phi}}2}\int_{[C]}H_{m}d\nu_{m}.
\end{eqnarray*}
This yields
\begin{equation}
\label{eq-hmequiconti}
|H_{m}(x)-H_{m}(x')|\le  (e^{\frac{C_{\phi}}2d(x,x')}-1)e^{ \frac{C_{\phi}}2},
\end{equation}
since $\disp 0\le \int_{[C]}H_{m}d\nu_{m}\le 1$. 

\end{proof}

Finally we finish this section with a result on the return-times: 

\begin{proposition}
\label{prop-returnmeas}
Let $r_{DA,m}$ (resp. $r_{AD,m}$) be the return time function in the hole restricted  on the green-cylinders in $\Om_{D}$ (resp. $\Om_{A}$). 
Then, 
$$\wh\mu_{m}(\Om_{A})=\dfrac{\int r_{DA,m}d\mu_{m}}{\int r_{DA,m}d\mu_{m}+\int r_{AD,m}d\mu_{m}},\text{ and }\wh\mu_{m}(\Om_{D})=\dfrac{\int r_{AD,m}d\mu_{m}}{\int r_{DA,m}d\mu_{m}+\int r_{AD,m}d\mu_{m}}.$$ 
\end{proposition}
\begin{proof}
We consider the natural extension $\S^{*}_{m}$ of $\S_{m}$. The canonical projection onto $\S_{m}$ is $\Pi$ and $\mu^{*}_{m}$ is the unique invariant probability such that $\Pi(\mu^{*}_{m})=\wh\mu_{m}$. Points in $\S^{*}_{m}$ are bi-infinite sequences $(x_{n})_{n\in \Z}$
 such that for any $k\in \Z$, $(x_{n})_{n\ge k}$ is in $\S_{m}$. 
 
 For $\ul x$ in $\S^{*}_{m}$, the  cylinder $C_{n}^{m}(\ul x)$ is the set of points $\ul y$ such that $y_{-n}=x_{-n},\ldots, y_{0}=x_{0},\ldots , y_{m}=x_{m}$.

 Note that $\wh\mu_{m}$ is a Gibbs measure, hence it has full support. This also holds for $\mu_{m}^{*}$ and then  $\mu_{m}^{*}$-almost every point returns infinitely time in the past and in the future in the hole. 

We consider a generic point, say $\ul x$ in $\Pi^{-1}(\Om_{A})$, with first  entrance-time in the hole for the forward orbit  $m$, and for the backward orbit,  $n$. 
Since $\s$ is one-to-one, all the  cylinders $C_{n+k}^{m+1-k}(\s^{k}(\ul x))$, $-n\le k\le m$ have the same $\mu_{m}^{*}$-measure. 
This measure is $\mu_{m}^{*}(C_{0}^{n+m+1}(\s^{-n}(\ul x)))=\wh\mu_{m}([x_{-n}\ldots x_{0}\ldots x_{m+1}])$, with $x_{-n}x_{-n+1}=\de_{i}\al_{k}$ and $x_{m}x_{m+1}=\al_{i}\de_{l}$. Therefore, in $\S_{m}$, $[x_{-n}\ldots x_{0}\ldots x_{m+1}]$ is a red-cylinder with $red\subset[\de_{i}\al_{k}]\to[\al_{i}\de_{l}]$. 

Furthermore, all the cylinders $C_{n+k}^{m+1-k}(\s^{k}(\ul x))$, $-n\le k\le m$ are disjoints, and the $n+m+1$ cylinders $C_{n+k}^{m+1-k}(\s^{k}(\ul x))$, with $-n+1\le k\le m$ are in $\Pi^{-1}(\Om_{A})$. We emphasize that for the considered red-cylinder, $n+m+1$ is the return time in the hole.

Doing the same work for all the red-cylinders in the hole in $\Om_{D}$ (with target in the hole in $\Om_{A}$), 
we get a cover of $\Pi^{-1}(\Om_{A})$, and then 
$$\mu_{m}^{*}(\Pi^{-1}(\Om_{A}))=\int_{Hole\subset \Om_{D}} r_{DA,m}\,d\wh\mu_{m}.$$
This yields $\wh\mu_{m}(\Om_{A})=\int_{Hole\subset \Om_{D}} r_{DA,m}\,d\wh\mu_{m}$. doing the same work exchanging the roles of $A$ and $D$ we get 
$$\wh\mu_{m}(\Om_{D})=\int_{Hole\subset \Om_{A}} r_{AD,m}\,d\wh\mu_{m}.$$
As $\wh\mu_{m}(\Om_{A})+\wh\mu_{m}(\Om_{D})=1$, we get the result. 
\begin{figure}[htbp]
\begin{center}
\includegraphics[scale=0.5]{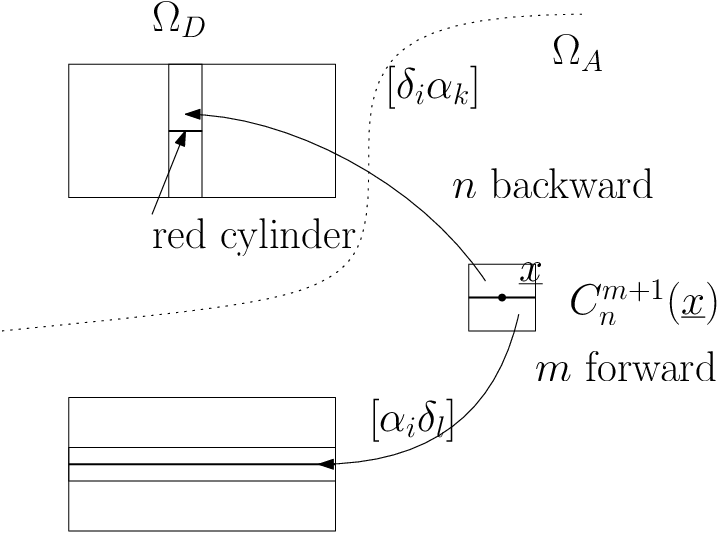}
\caption{{\bf Partition of $\Om_{A}$}}
\label{fig-partition-return}
\end{center}
\end{figure}

\end{proof}

\section{Proof of Theorem \ref{th-main1-2}}\label{sec-proof-main-1-2}

\subsection{Convergence for $\nu_{m}$}

\subsubsection{Loops, strings of transfer operators}
We will do computations that involve the returns in the holes in $\Om_{A}$ or $\Om_{D}$. The goal is to get estimates on $\nu_{m}$ and $H_{m}$. 
The spirit of the computation is simple, see Figure \ref{fig-trajectoCLA}, despite the writing is a littler bit heavy. 

\begin{figure}[htbp]
\begin{center}
\includegraphics[scale=0.5]{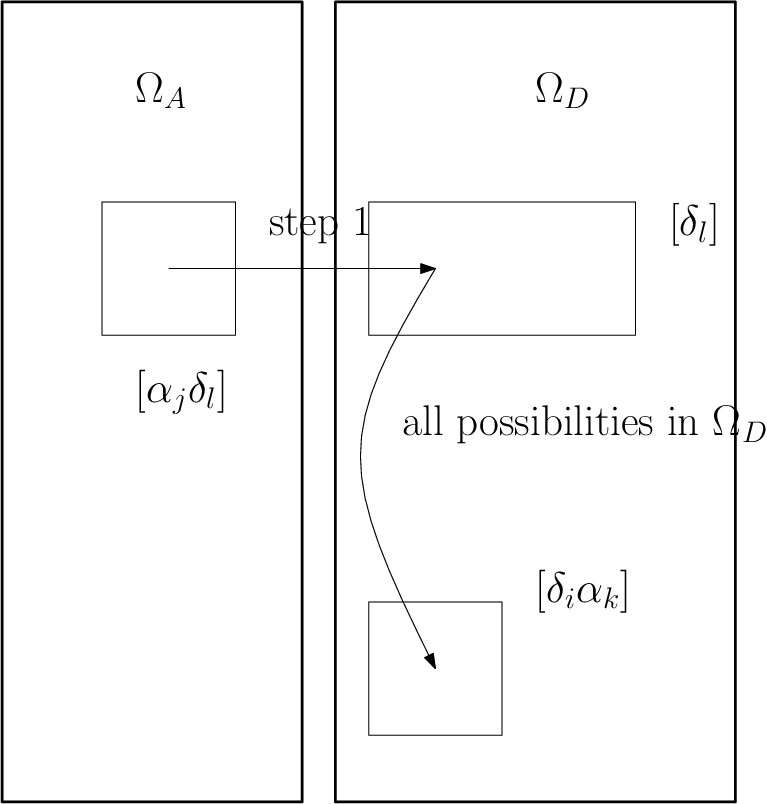}
\caption{{\bf Orbits leaving the hole in $\Om_{A}$ to return in the hole in $\Om_{D}$}}
\label{fig-trajectoCLA}
\end{center}
\end{figure}

One round runs as follows: an orbit leaves the hole in $\Om_{A}$ and arrives in $\Om_{D}$. Then is stays here for a time larger than $n'_{m}$, to eventually arrive in some green-cylinder $[\de_{i}\al_{k}]$. Then, starts the second round. 

\medskip
Estimates are obtained by estimating the transfer operator for these kind of orbits, that means the induced transfer operator for the first or the second return in the hole. 

The point is that, in spirit, these orbits see two special transitions (when they enter in the hole), and otherwise only see the dynamics in $\S_{D'}$, $\S_{D}$ or $\S_{A'}$ and $\S_{A}$.

We are thus led to estimate quantities of the form 
$$\CL^{r}_{A}\circ \CL^{n_{m}-1}_{A'}\circ T_{D\to A}\circ \CL_{D}^{n}\circ \CL_{D'}^{n'_{m}-1}\circ T_{A\to D},$$
where $T_{.\to.}$ are the transfer operators for the transitions (that is in the hole). 

In such estimates, we apply formulas ,\eqref {equ1-cla-extend} ,\eqref {equ1-claprime-extend} \eqref {equ1-cldprime-extend}  and \eqref {equ1-cld-extend} to compute the dominating terms and then sum over all the possibilities ($n$ and $r$).

\medskip
In the following, $P_{A'}$ and $P_{D'}$ respectively denote the pressure for $\S_{A'}$ and $\S_{D'}$ for the potential $\phi$. They are the $\log$ of the spectral radius of the transfer operators $\CL_{A'}$ and $\CL_{D'}$.
Because $A'<A$ and $D'>D$, uniqueness of the equilibrium state for irreducible subshifts  yields 

$$\boxed{P_{A'}<P} \text{ and }\boxed{P_{D'}<P}.$$

%
%
%

\subsubsection{Computations for $\nu_{m}$ with the first return in the hole}

In this first computation we consider the first return in the hole. In the following, $f$ is a Lipschitz-continuous function defined in the hole in $\Om_{A}$. The key point is to use that $\nu_{m}$ is a conformal measure (see Equation \eqref{eq1-conform}). This  makes the transfer operator appearing. 

\begin{align}
\int f\BBone_{[\al_{j}\de_{l}]}d\nu_{m}&= \sum_{[\de_{i}\al_{k}]}\sum_{n}\sum_{\de}\int f\BBone_{[\al_{j}\de_{l}\de^{n'_{m}-2}\de^{n}\de_{i}\al_{k}]}d\nu_{m}\nonumber\\
\intertext{ we use conformality}
&=\sum_{[\de_{i}\al_{k}]}\sum_{n}\sum_{\de}e^{-(n'_{m}+n)P_{m}}\int f(\al_{j}\de_{l}\de^{n'_{m}-2}\de^{n}*)e^{S_{n'_{m}+n}(\phi)(\al_{j}\de_{l}\de^{n'_{m}-2}\de^{n}*)}\BBone_{[\de_{i}\al_{k}]}(*)\,d\nu_{m}(*)\nonumber\\
\intertext{ we make $\CL_{D}^{n}$ appear}
&= \sum_{[\de_{i}\al_{k}]}\sum_{\de}e^{-(n'_{m}+n)P_{m}}\int \sum_{n} \CL_{D}^{n}\left(f(\al_{j}\de_{l}\de^{n'_{m}-2}*)e^{S_{n'_{m}}(\phi)(\al_{j}\de_{l}\de^{n'_{m}-2}*)}\right)\BBone_{[\de_{i}\al_{k}]}(*)\,d\nu_{m}(*)\nonumber\\
\intertext{we use \eqref{equ1-cla-extend}}
&= \sum_{[\de_{i}\al_{k}]}\sum_{\de}e^{-(n'_{m}+n)P_{m}}\int \sum_{n}e^{nP}\left(\left[
\int f(\al_{j}\de_{l}\de^{n'_{m}-2}*)e^{S_{n'_{m}}(\phi)(\al_{j}\de_{l}\de^{n'_{m}-2}*)}d\nu_{D}\right]\wt H_{D}\BBone_{[\de_{i}\al_{k}]}\right.\nonumber\\
&\left. +e^{-n\rho}\Psi_{D}^{n}\left(f(\al_{j}\de_{l}\de^{n'_{m}-2}*)e^{S_{n'_{m}}(\phi)(\al_{j}\de_{l}\de^{n'_{m}-2}*)}\right)\BBone_{[\de_{i}\al_{k}]}
\right)d\nu_{m}\label{toto1-bisrepetita}
\end{align}
Now, we focus on $\disp \sum_{\de}f(\al_{j}\de_{l}\de^{n'_{m}-2}*)e^{S_{n'_{m}}(\phi)(\al_{j}\de_{l}\de^{n'_{m}-2}*)}$ (that is evaluated at points in $\Om_{D}$). 

\begin{align*}
 \disp \sum_{\de}f(\al_{j}\de_{l}\de^{n'_{m}-2}*)e^{S_{n'_{m}}(\phi)(\al_{j}\de_{l}\de^{n'_{m}-2}*)}&= \CL_{D'}^{n'_{m}-1}\left(f(\al_{j}*)e^{\phi(\al_{j}*)}\right)\\
 \intertext{ we use \eqref{equ1-claprime-extend}}
 &\hskip -4cm= e^{(n'_{m}-1)P_{D'}}\left(\left[\int f(\al_{j}*)e^{\phi(\al_{j}*)}d\nu_{D'}\right]\wt H_{D'}\right.
 \left.+e^{-(n'_{m}-1)\rho}
\Psi^{n_{m}-1}_{D'}(f(\al_{j}*)e^{\phi(\al_{j}*)})\right).\\
 \end{align*}
If we re-inject this into \ref{toto1-bisrepetita} we find that $\int f\BBone_{[\al_{j}\de_{l}]}d\nu_{m}$ is equal to the sum of the next 4 summands: 
$$\sum_{[\de_{i}\al_{k}]}\sum_{n}e^{n(P-P_{m})}e^{-P_{m}}e^{(n'_{m}-1)(P_{D'}-P_{m})}\left[\int f(\al_{j}*)e^{\phi(\al_{j}*)}d\nu_{D'}\right]\left[\int \wt H_{D'}d\nu_{D}\right]\left[\int \wt H_{D}\BBone_{[\de_{i}\al_{k}]}d\nu_{m}\right]$$

$$\sum_{[\de_{i}\al_{k}]}\sum_{n}e^{n(P-P_{m})}e^{-P_{m}}e^{(n'_{m}-1)(P_{D'}-P_{m}-\rho)}\left[\int \Psi^{n_{m}-1}_{D'}(f(\al_{j}*)e^{\phi(\al_{j}*)})d\nu_{D}\right]\left[\int \wt H_{D}\BBone_{[\de_{i}\al_{k}]}d\nu_{m}\right]$$

$$\sum_{[\de_{i}\al_{k}]}\sum_{n}e^{n(P-P_{m}-\rho)}e^{-P_{m}}e^{(n'_{m}-1)(P_{D'}-P_{m})}\left[\int f(\al_{j}*)e^{\phi(\al_{j}*)}d\nu_{D'}\right]\left[\int \wt \Psi^{n}_{D}(H_{D'})\BBone_{[\de_{i}\al_{k}]}d\nu_{m}\right]$$

$$\sum_{[\de_{i}\al_{k}]}\sum_{n}e^{n(P-P_{m}-\rho)}e^{-P_{m}}e^{(n'_{m}-1)(P_{D'}-P_{m}-\rho)}\left[\int \wt \Psi^{n}_{D}(\Psi^{n_{m}-1}_{D'}(f(\al_{j}*)e^{\phi(\al_{j}*)}))\BBone_{[\de_{i}\al_{k}]}d\nu_{m}\right]$$

We remind that $\wh H_{D}$ and $\wh H_{D'}$ are bounded from below away from 0 and from above. Furhtermore, $\Psi_{D}$ and $\Psi_{D'}$ have spectral radius 1. Then, as $m$ goes to $+\8$, the first summand behaves like $\asymp \dfrac{e^{-P}e^{n'_{m}(P_{D'}-P_{m})}}{P_{m}-P}$. The second summand behaves like $\asymp\dfrac{e^{-P}e^{n'_{m}(P_{D'}-P_{m}-\rho)}}{P_{m}-P}$. The third summand behaves like $\asymp e^{n'_{m}(P_{D'}-P_{m})}$ and the fourth summand goes to zero exponentially faster that $e^{n'_{m}(P_{D'}-P)}$. This shows that the dominating term is the first summand. 

We remind that the union of the $[\de_{i}\al_{k}]$ is the set of green cylinders in $\Om_{D}$. Similarly the union of the $[\al_{j}\de_{l}]$ is the set of the green cylinders in $\Om_{A}$. The measure $\nu_{m}$ has its support in these green cylinders.

Hence, we get 
\begin{equation}
\label{equ1-num1loopnew}
\int_{\Om_{A}} fd\nu_{m}= \dfrac{e^{-P}e^{(n'_{m}-1)(P_{D'}-P_{m})}}{P_{m}-P}\left[\sum_{[\al_{j}\de_{l}]}\int f(\al_{j}*)e^{\phi(\al_{j}*)}d\nu_{D'}\right]\left[\int \wt H_{D'}d\nu_{D}\right]\left[\int_{\Om_{D}} \wt H_{D}d\nu_{m}\right](1+o(1))
\end{equation}

\subsubsection{Computation for $\nu_{m}$ with the second return in the hole}

The same computation exchanging the roles of $A$ and $D$, yields for $\wt f:\Om_{D}\to\R$ Lipschitz-continuous:
$$\int_{\Om_{D}} \wt fd\nu_{m}= \dfrac{e^{-P}e^{(n_{m}-1)(P_{A'}-P_{m})}}{P_{m}-P}\left[\sum_{[\de_{i}\al_{k}]}\int \wt f(\de_{i}*)e^{\phi(\de_{i}*)}d\nu_{A'}\right]\left[\int \wt H_{A'}d\nu_{A}\right]\left[\int_{\Om_{A}} \wt H_{A}d\nu_{m}\right](1+o(1))$$
If we replace $\wt f$ by $\wt H_{D}$ and then replace in \eqref{equ1-num1loopnew} we get 

\begin{align}
\int_{\Om_{A}} fd\nu_{m}&= \dfrac{e^{-2P}e^{(n'_{m}-1)(P_{D'}-P_{m})}e^{(n_{m}-1)(P_{A'}-P_{m})}}{(P_{m}-P)^{2}}\left[\sum_{[\al_{j}\de_{l}]}\int f(\al_{j}*)e^{\phi(\al_{j}*)}d\nu_{D'}\right]\left[\int \wt H_{D'}d\nu_{D}\right]\nonumber\\
& \left[\sum_{[\de_{i}\al_{k}]}\int \wt H_{D}(\de_{i}*)e^{\phi(\de_{i}*)}d\nu_{A'}\right]\left[\int \wt H_{A'}d\nu_{A}\right]\left[\int_{\Om_{A}} \wt H_{A}d\nu_{m}\right](1+o(1))
\label{equ1-num2loopnew}
\end{align}

\begin{definition}
\label{def-mespreima2}
We denote by $\al^{*}\nu_{D'}$ (resp. $\de^{*}\nu_{A'}$) the measure defined on the  green cylinders in $\Om_{A}$ (resp. $\Om_{D}$) whose image by $\s$ is $\nu_{D'}$ (resp. $\nu_{A'}$).

We also set $\disp \l_{m}:=\dfrac{e^{(n_{m}-1)(P_{A'}-P_{m})}}{1- e^{-(P_{m}-P)}}$ and $\l'_{m}:=\dfrac{e^{(n'_{m}-1)(P_{D'}-P_{m})}}{1- e^{-(P_{m}-P)}}$.
\end{definition}
 Then 
\eqref{equ1-num2loopnew} is equivalent to 
\begin{equation}
\int_{\Om_{A}} fd\nu_{m}= e^{-2P}\l_{m}\l'_{m}\left[\int fe^{\phi}d\al^{*}\nu_{D'}\right]\left[\int \wt H_{D'}d\nu_{D}\right] \left[\int \wt H_{D}e^{\phi}d\de^{*}\nu_{A'}\right]\left[\int \wt H_{A'}d\nu_{A}\right]\left[\int_{\Om_{A}} \wt H_{A}d\nu_{m}\right](1+o(1))
\label{equ1-num3loopnew}
\end{equation}

\subsubsection{Limits for $\nu_{m}$}

\begin{proposition}
\label{prop2-cvnumgenemain1}
The quantity $\L_{m}:=\l_{m}\l'_{m}$ goes to 
$$\L:=\dfrac{e^{2P}}{\left[\int \wt H_{A}e^{\phi}d\al^{*}\nu_{D'}\right]\left[\int \wt H_{D'}d\nu_{D}\right] \left[\int \wt H_{D}e^{\phi}d\de^{*}\nu_{A'}\right]\left[\int \wt H_{A'}d\nu_{A}\right]}$$ as $m$ goes to $+\8$. 

The  probability measure $\nu_{m\,|A}$ (resp. $\nu_{m\,|D}$) converges for the weak* topology to the probability measure  proportional to $e^{\phi}.\al^{*}\nu_{D'}$ (resp. $e^{\phi}.\al^{*}\nu_{A'}$).  

Furthermore 
\begin{equation}
\label{equ3-rapportnum2}
\dfrac1{\l'_{m}}\dfrac{\nu_{m}(\Om_{A})}{\nu_{m}(\Om_{D})}\to_{m\to+\8}e^{-P}\dfrac{\int_{\Om_{A}}e^{\phi}\,d\al^{*}\nu_{D'}}{\int_{\Om_{D}}e^{\phi}\,d\de^{*}\nu_{A'}}\int \wt H_{D'}d\nu_{D}\int_{\Om_{D}} \wt H_{D}e^{\phi}\,d\de^{*}\nu_{A'}.
\end{equation}

\end{proposition}
\begin{proof}
If we replace $f$ by $\wt H_{A}$ in \eqref{equ1-num3loopnew} and do $m\to+\8$ we get 
$$\lim_{\minf}\l_{m}\l'_{m}=\dfrac{e^{2P}}{\left[\int \wt H_{A}e^{\phi}d\al^{*}\nu_{D'}\right]\left[\int \wt H_{D'}d\nu_{D}\right] \left[\int \wt H_{D}e^{\phi}d\de^{*}\nu_{A'}\right]\left[\int \wt H_{A'}d\nu_{A}\right]}.$$
Dividing in both side of \eqref{equ1-num3loopnew} by $\nu_{m}(\Om_{A})$, we see that the unique possible accumulation point for $\nu_{m|A}$ is the probability measure equivalent to $e^{\phi}. \al^{*}\nu_{D'}$. Result for $\nu_{m|D}$ is obtained by symmetry. 

Finally, if we do $f=1$ in \eqref{equ1-num1loopnew}, then divide by $\nu_{m}(\Om_{D})$ and make $m$ go to $+\8$, we get \eqref{equ3-rapportnum2}.
\end{proof}

\subsection{Convergences for $H_{m}$}
In this subsection, we prove that ``in spirit'' $H_{m}$ converges as $m$ goes to $+\8$.  Nevertheless, the convergence depends on how $\l'_{m}$ and $\l_{m}$ behave, in particular if they are bounded or not. 

We remind that $\l'_{m}\l_{m}$ does converge, hence, if one, say $\l'_{m}$, goes to 0, then $\l_{m}$ goes to $+\8$. 

In that case one needs to consider the correct normalization, $\dfrac1{\l_{m}}H_{m}$ but only on one part of the hole. 

The strategy is to prove that there exists a unique accumulation point for $H_{m}$ (with the adapted normalization).  

\subsubsection{A general computation for $H_{m}$}

We pick $x\in [\de_{i}\al_{k}]$. Then we get 
\begin{align}
H_{m}(x)&=\sum_{[\al_{j}\de_{l}]}\sum_{n=0}^{+\8}\sum_{\de^{n},\de^{n'_{m}-2}}e^{-(n+n'_{m})P_{m}}e^{S_{n+2}(\phi)(\al_{j}\de_{l}\de^{n'_{m}-2}\de^{n} x)}H_{m}(\al_{j}\de_{l}\de^{n'_{m}-2}\de^{n} x)\nonumber\\
&= \sum_{[\al_{j}\de_{l}]}\sum_{\de^{n'_{m}-2}}\sum_{n=0}^{+\8}e^{-(n+n'_{m})P_{m}}\CL_{D}^{n}(e^{\phi(\al_{j}\de_{l}\de^{n'_{m}-2}*)}H_{m}(\al_{j}\de_{l}\de^{n'_{m}-2}*))(x)\nonumber\\
&\hskip -4cm= \sum_{[\al_{j}\de_{l}]}\sum_{\de^{n'_{m}-2}}e^{-n'_{m}P_{m}}\sum_{n=0}^{+\8}e^{n(P-P_{m})}
\left(\left[\int e^{\phi(\al_{j}\de_{l}\de^{n'_{m}-2}*)}H_{m}(\al_{j}\de_{l}\de^{n'_{m}-2}*)d\nu_{D}\right]\wt H_{D}(x)+e^{-n\rho}\Psi_{D}(e^{\phi(\al_{j}\de_{l}\de^{n'_{m}-2}*)}H_{m}(\al_{j}\de_{l}\de^{n'_{m}-2}*))
\right)\nonumber\\
&= \sum_{[\al_{j}\de_{l}]}e^{-n'_{m}P_{m}}\sum_{n=0}^{+\8}e^{n(P-P_{m})}
\left(\left[\int \sum_{\de^{n'_{m}-2}}e^{\phi(\al_{j}\de_{l}\de^{n'_{m}-2}*)}H_{m}(\al_{j}\de_{l}\de^{n'_{m}-2}*)d\nu_{D}\right]\wt H_{D}(x)\right.\nonumber\\
&\left.+e^{-n\rho}\Psi_{D}(\sum_{\de^{n'_{m}-2}}e^{\phi(\al_{j}\de_{l}\de^{n'_{m}-2}*)}H_{m}(\al_{j}\de_{l}\de^{n'_{m}-2}*))
\right)\nonumber\\
&= \sum_{[\al_{j}\de_{l}]}e^{-n'_{m}P_{m}}\sum_{n=0}^{+\8}e^{n(P-P_{m})}
\left(\left[\int \CL_{D'}^{n'_{m}-1}\left(e^{\phi(\al_{j}*)}\BBone_{[\de_{l}]}(*)H_{m}(\al_{j}*)\right)d\nu_{D}\right]\wt H_{D}(x)\right.\nonumber\\
&\left.+e^{-n\rho}\Psi_{D}^{n}\left(\CL_{D'}^{n'_{m}-1}\left(e^{\phi(\al_{j}*)}\BBone_{[\de_{l}]}(*)H_{m}(\al_{j}*)\right)\right)
\right)\nonumber\\
& \hskip -2cm=e^{-P_{m}}\l'_{m}\wt H_{D}(x)\left[\int \left[\int_{\Om_{A}} e^{\phi}H_{m}d(\al^{*}\nu_{D'} )\right]\wt H_{D'}
+e^{-(n'_{m}-1)\rho}\Psi_{D'}^{n'_{m}-1}\left(\sum_{[\al_{j}\de_{l}]}
e^{\phi(\al_{j}*)}\BBone_{[\de_{l}]}(*)H_{m}(\al_{j}*)
\right)d\nu_{D}
\right]\nonumber\\
&+e^{-P_{m}}e^{(n'_{m}-1)(P_{D'}-P_{m})}
\sum_{n=0}^{+\8}e^{n(P-P_{m}-\rho)}
\Psi_{D}^{n}\left(
\left[\int_{\Om_{A}} e^{\phi}H_{m}d(\al^{*}\nu_{D'} )\right]\wt H_{D'}+e^{-(n'_{m}-1)\rho}\Psi_{D'}^{n'_{m}-1}(e^{\phi}H_{m})
\right)
.\label{eq1-hmgene2}
\end{align}
Again we get four summands. The third and the fourth respectively are in $o$  of the first and the second since $P-P_{m}$ goes to $0$ as $m$ goes to $+\8$ and we get a multiplicative quantity $\dfrac1{1-e^{P_{m}-P}}$ in the first and the second compared to the third and the fourth. 

Now, $\nu_{D'}$ gives positive weight to any open set in $D'$, hence $\al^{*}\nu_{D'}$ gives positive weight to any open set in the green cylinders in $\Om_{A}$. Equation \eqref{eq-hmequiconti} shows that all the $H_{m}$ are equicontinuous. We can thus fix some open ball with fixed radius (independent from $m$) such that on that ball $H_{m}$ is ``almost'' equal to its maximum (on $\Om_{A}$). All these balls have a bounded away from zero $\al^{*}\nu_{D'}$-measure as $m$ changes. This yields that  the ratio $\disp \dfrac{\int e^{\phi}H_{m}d(\al^{*}\nu_{D'})}{||H_{m}||_{\8}}$ is bounded away from zero. This yields that the second summand in \eqref{eq1-hmgene2} is in $o$ of the first one as $m$ goes to $+\8$ since $n'_{m}$ goes to $+\8$. 

Hence we get for every $x$ in any green cylinder in $\Om_{D}$:
\begin{equation}
H_{m}(x)=e^{-P_{m}}\l'_{m}\wt H_{D}(x)\left[\int\wt H_{D'}d\nu_{D}\right] \left[\int_{\Om_{A}} e^{\phi}H_{m}d(\al^{*}\nu_{D'} )\right](1+\eps(m)),
\label{eq1-hmgene3}
\end{equation}
with $\eps(m)\to_{\minf}0$.
Similarly, we get  for every $x$ in any green cylinder in $\Om_{A}$:
\begin{equation}
H_{m}(x)=e^{-P_{m}}\l_{m}\wt H_{A}(x)\left[\int\wt H_{A'}d\nu_{A}\right] \left[\int_{\Om_{D}} e^{\phi}H_{m}d(\de^{*}\nu_{A'} )\right](1+\eps(m)),
\label{eq1-hmgene4}
\end{equation}
with $\eps(m)\to_{\minf}0$.

\subsubsection{Convergence in the case $\l_{m}\to+\infty$}

\begin{proposition}
\label{prop-cvHmcasgenegene}
$H_{m|\Om_{D}}$ and $\dfrac{1}{\l_{m}}H_{m|\Om_{A}}$ respectively converge to $K_{D}\wt H_{D}$ and $K_{A}\wt H_{A}$, with $K_{D}=\dfrac{\int e^{\phi}d\de^{*}\nu_{A'}}{2\int \wt H_{D}e^{\phi}d\de^{*}\nu_{A'}}$ and $K_{A}=\dfrac{\int e^{\phi}d\de^{*}\nu_{A'}\int \wt H_{A'}d\nu_{A}}{2e^{P}}$.
\end{proposition}
The rest of the subsection is devoted to the proof of the proposition.

We consider a subsequence such that $\l_{m_{k}}\to+\8$. For simplicity we shall always write $\l_{m}\to+\8$. As it is said above,  $\l'_{m}$ goes to 0. 

Furthermore \eqref{equ3-rapportnum2} yields 
$$\nu_{m}(\Om_{A})=\dfrac{K(1+o(1))\l'_{m}}{1+K(1+o(1))\l'_{m}},$$
with $\boxed{K=e^{-P}\dfrac{\int_{\Om_{A}}e^{\phi}\,d\al^{*}\nu_{D'}}{\int_{\Om_{D}}e^{\phi}\,d\de^{*}\nu_{A'}}\int\wt H_{D'}d\nu_{D}\int_{\Om_{D}} \wt H_{D}e^{\phi}\,d\de^{*}\nu_{A'}}$, and  then 

\begin{equation}
\label{equ1-lmnum2}
\l_{m}\nu_{m}(\Om_{A})=\dfrac{K(1+o(1))\L_{m}}{1+K(1+o(1))\l'_{m}}\to_{\minf}K\L.
\end{equation}
This means that $\nu_{m}(\Om_{A})$ goes to 0 ans then $\nu_{m}(\Om_{D})$ goes to 1. 

\bigskip

If $[C]$ is a green cylinder, \eqref{equ1disto} yields 
$$||H_{m|[C]}||_{\8}\nu_{m}([C])\le e^{\frac{C_{\phi}}{4}}\int_{[C]}H_{m}\,d\nu_{m}.$$

As it is said above, $\nu_{D'}$ gives positive weight to any open set in $D'$, thus $\al^{*}\nu_{D'}$ gives positive weight to any open set in the green cylinders in $\Om_{A}$. Then, Prop. \ref{prop2-cvnumgenemain1} shows that $\nu_{m|A}([C])$ is bounded away from zero as $m$ goes to $+\8$. 

Considering some green-cylinder in $\Om_{A}$ where $H_{m}$ is maximal, \eqref{equ1-lmnum2}  yields

\begin{equation}
\label{equ1-majolmHmAgenegene}
\dfrac1{\l_{m}}||H_{m|\Om_{A}}||_{\8}\le \dfrac{e^{\frac{C_{\phi}}{4}}}{\l_{m}\nu_{m}(\Om_{A})\nu_{m|A}([C])}\le \wt C,
\end{equation}
for some universal constant $\wt C$. 
This plus \eqref{eq-hmequiconti} yields that the family $\dfrac1\l_{m}H_{m|\Om_{A}}$ is bounded for the Lipschitz norm thus, we can pick some converging subsequence. 

On the other hand, doing the  same reasoning for $D$  we get 
\begin{equation}
\label{equ1-majolmHmDgenegene}
||H_{m|\Om_{D}}||_{\8}\le \dfrac{e^{\frac{C_{\phi}}{4}}}{\nu_{m}(\Om_{D})\nu_{m|D}([C])}\le C'
\end{equation}
for some universal constant $C'$.
 Again, \eqref{eq-hmequiconti} yields that the family $H_{m|\Om_{D}}$ is bounded for the Lipschitz norm thus we can pick some converging subsequence.

Now, \eqref{eq1-hmgene3} and \eqref{eq1-hmgene4} show that $\dfrac1{\l_{m}}H_{m|\Om_{A}}$ converges\footnote{up to the correct subsequence} to $K_{A}\wt H_{A}$ while $H_{m|\Om_{D}}$ converges to $K_{D}\wt H_{D}$, with 
\begin{equation}
\label{eq1-KAKDgenegene}
\dfrac{K_{D}}{K_{A}}=\dfrac{e^{P}}{\int \wt H_{A'}d\nu_{A}\int \wt H_{D}e^{\phi}d\de^{*}\nu_{A'}}
\end{equation}

The normalization condition yields 
\begin{align*}
 1&=\int H_{m}\,d\nu_{m}=\l_{m}\int_{\Om_{A}}\dfrac1{\l_{m}}H_{m}\,d\nu_{m}+\int_{\Om_{D}}H_{m}\,d\nu_{m}\\
 &= \l_{m}\nu_{m}(\Om_{A})\int_{\Om_{A}}\dfrac1{\l_{m}}H_{m}\,d\nu_{m|A}+\nu_{m}(\Om_{D})\int_{\Om_{D}}H_{m}\,d\nu_{m|D},\\
 \intertext{ and doing $m\to+\8$ we get }
 1&=K.\Lambda. K_{A}\dfrac{\int \wt H_{A}e^{\phi}d\al^{*}\nu_{D'}}{\int e^{\phi}d\al^{*}\nu_{D'}}+K_{D}\dfrac{\int \wt H_{D}e^{\phi}d\de^{*}\nu_{A'}}{\int e^{\phi}d\de^{*}\nu_{A'}}.
\end{align*}
Therefore, \eqref{eq1-KAKDgenegene} yields $\boxed{K_{A}=\dfrac{\int e^{\phi}d\de^{*}\nu_{A'}\int \wt H_{A'}d\nu_{A}}{2e^{P}}}$  and $\boxed{K_{D}=\dfrac{\int e^{\phi}d\de^{*}\nu_{A}}{2\int \wt H_{D}e^{\phi}d\de^{*}\nu_{A'}}}$.

This finally show that $H_{m|\Om_{D}}$ and $\dfrac{1}{\l_{m}}H_{m|\Om_{A}}$ actually converge since they have only one possible accumulation point. 

\subsubsection{Convergence in the case $\l_{m}\nrightarrow0,+\infty$}
Now, we assume that $\l_{m}$ (and thus this also holds for $\l'_{m}$) is bounded away from 0 and from above. 

We consider a subsequence such that $\l_{m}$ goes to $\l$ and $\l'_{m}$ goes to $\l'$. We have $\L=\l\l'$. 
Then \eqref{equ3-rapportnum2} yields 
$$\nu_{m}(\Om_{A})\to_{\minf} \dfrac{\l'K}{1+\l'K}\text{ and }\nu_{m}(\Om_{D})\to_{m\to+\8}\dfrac1{1+\l'.K}. $$

The same arguments as previously show that \eqref{eq1-hmgene3} and \eqref{eq1-hmgene4} still hold, since in that case $\l_{m}$ goes to $\l\in ]0,+\8[$ and $\nu_{m}(\Om_{A})$ and $\nu_{m}(\Om_{D})$ converge to positive real numbers. 
This yields that $H_{m}$ is uniformly bounded in $\Om$, and \eqref{eq-hmequiconti} shows that one can consider accumulation points for the   $||\ ||_\8$-norm.

Then, \eqref{eq1-hmgene3} and \eqref{eq1-hmgene4} show that $H_{m|\Om_{A}}$ converges to $K_{A}\wt H_{A}$ while $H_{m|\Om_{D}}$ converges to $K_{D}\wt H_{D}$, with 
\begin{equation}
\label{eq2-KAKDgenegene}
\dfrac{K_{D}}{K_{A}}=\dfrac{e^{P}}{\l\int \wt H_{A'}d\nu_{A}\int \wt H_{D}e^{\phi}d\de^{*}\nu_{A'}}
\end{equation}

Furthermore, the normalization condition yields:
\begin{align*}
 1&=\int H_{m}\,d\nu_{m}=\int_{\Om_{A}}H_{m}\,d\nu_{m}+\int_{\Om_{D}}H_{m}\,d\nu_{m}\\
 &= \nu_{m}(\Om_{A})\int_{\Om_{A}}H_{m}\,d\nu_{m|A}+\nu_{m}(\Om_{D})\int_{\Om_{D}}H_{m}\,d\nu_{m|D},\\
 \intertext{ and doing $m\to+\8$ we get }
 1&=\frac{\l'K}{1+\l'K}K_{A}\dfrac{\int \wt H_{A}e^{\phi}d\al^{*}\nu_{D}}{\int e^{\phi}d\al^{*}\nu_{D}}+K_{D}\dfrac{\int \wt H_{D}e^{\phi}d\de^{*}\nu_{A}}{\int e^{\phi}d\de^{*}\nu_{A}}\dfrac1{1+\l'K}\\
 \intertext{hence}
 1&= \frac{\l'K}{1+\l'K}K_{D}e^{-P}\l\int \wt H_{A'}d\nu_{A}\int\wt H_{D}e^{\phi}\,d\de^{*}\nu_{A}\dfrac{\int \wt H_{A}e^{\phi}d\al^{*}\nu_{D}}{\int e^{\phi}d\al^{*}\nu_{D}}
 +K_{D}\dfrac{\int \wt H_{D}e^{\phi}d\de^{*}\nu_{A}}{\int e^{\phi}d\de^{*}\nu_{A}}\dfrac1{1+\l'K}\\
 &=K_{D}\dfrac{\int \wt H_{D}e^{\phi}d\de^{*}\nu_{A}}{\int e^{\phi}d\de^{*}\nu_{A}}\dfrac1{1+\l'K}\left(e^{-P}\l\l'K\int \wt H_{A'}d\nu_{A}\int\wt H_{A}e^{\phi}\,d\al^{*}\nu_{D}\dfrac{\int e^{\phi}d\de^{*}\nu_{A}}{\int e^{\phi}d\al^{*}\nu_{D}}+1
\right)\\
\intertext{ we replace $\l\l'=\L$ and $K$ by their value in the bracket}
1&=2K_{D}\dfrac{\int \wt H_{D}e^{\phi}d\de^{*}\nu_{A}}{\int e^{\phi}d\de^{*}\nu_{A}}\dfrac1{1+\l'K}.
\end{align*}
Finally we get 
\begin{equation}
\label{equ1-KDbornegenegene}
K_{D}=\dfrac{\int e^{\phi}d\de^{*}\nu_{A}+e^{-P}\l'\int_{\Om_{A}}e^{\phi}\,d\al^{*}\nu_{D}\int \wt H_{D'}d\nu_{D}\int \wt H_{D}e^{\phi}\,d\de^{*}\nu_{A}}{2\int \wt H_{D}e^{\phi}\,d\de^{*}\nu_{A}}.
\end{equation}
\begin{equation}
\label{equ1-KAbornegenegene}
\hskip-2cmK_{A}=K_{D}e^{-P}\l\int\wt H_{D}e^{\phi}\,d\de^{*}\nu_{A}\int\wt H_{A'}d\nu_{A}=\dfrac{\l\int\wt H_{A'}d\nu_{A}\int e^{\phi}d\de^{*}\nu_{A}+\l\l'\int_{\Om_{A}}e^{\phi}\,d\al^{*}\nu_{D}\int \wt H_{D'}d\nu_{D}\int \wt H_{A'}d\nu_{A}\int \wt H_{D}e^{\phi}\,d\de^{*}\nu_{A}}{2e^{P}}.
\end{equation}
\begin{remark}
\label{rem-kakdcoherent}
We emphasize that these equalities are coherent with the ones obtained in the case $\l=+\8$
 and $\l'=0$. This is immediate for $K_{D}$. For $K_{A}$, in the case $\l=+\8$ the correct normalization is $\dfrac1{\l_{m}}H_{m|\Om_{A}}$ hence we have to divide by $\l$ in \eqref{equ1-KAbornegenegene} before replacing $\l'$ by 0. 
 $\blacksquare$\end{remark}

\subsection{Computation for $r_{DA,m}$}
\label{subsec-computereturn}
\subsubsection{A general computation}

\begin{eqnarray*}
\int r_{DA,m}d\mu_{m}&=& \sum_{n\ge 0}\sum_{[\de_{i}\al_{k}\al^{n_{m}-2}\al^{n}\al_{j}\de_{l}]}\int (n+n_{m})H_{m}\BBone_{[\de_{i}\al_{k}\al^{n_{m}-2}\al^{n}\al_{j}\de_{l}]}d\nu_{m}\\
&=& \sum_{n\ge 0}\sum_{[\de_{i}\al_{k}]}e^{-(n+n_{m})P_{m}}(n+n_{m})\int\CL_{A}^{n}(\CL_{A'}^{n_{m}-1}(H_{m}(\de_{i}*)e^{\phi(\de_{i}*)}\BBone_{[\al_{k}]}))d\nu_{m}\\
&&\text{ where we used conformality for $\nu_{m}$}\\
&=& \sum_{n\ge 0}\sum_{[\de_{i}\al_{k}]}e^{-(n+n_{m})P_{m}}(n+n_{m})e^{-nP}\int\left(\int_{[\al_{k}]} \CL_{A'}^{n_{m}-1}(H_{m}(\de_{i}*)e^{\phi(\de_{i}*)})d\nu_{A}\wt H_{A}\right.\\
&&+\left.e^{-n\rho}\Psi_{A}(\CL_{A'}^{n_{m}-1}(H_{m}(\de_{i}*)e^{\phi(\de_{i}*)}\BBone_{[\al_{k}]}))\right)d\nu_{m}.
\end{eqnarray*}

As before, we get 
\begin{align*}
\hskip-2cm \CL_{A'}^{n_{m}-1}(H_{m}(\de_{i}*)e^{\phi(\de_{i}*)}\BBone_{[\al_{k}]})&= e^{(n_{m}-1)P_{A'}}\left(\int H_{m}(\de_{i}*)e^{\phi(\de_{i}*)}\BBone_{[\al_{k}]}d\nu_{A'}\wt H_{A'}+e^{-(n_{m}-1)\rho}\Psi_{A'}(H_{m}(\de_{i}*)e^{\phi(\de_{i}*)}\BBone_{[\al_{k}]})\right)\\
 \intertext{ that we can rewrite as}
 &= (1+\eps(m))e^{(n_{m}-1)P_{A'}}\left[\int H_{m}e^{\phi}\BBone_{[\de_{i}\al_{k}]}d\de^{*}\nu_{A'}\right]\wt H_{A'}
 \end{align*}
with $\eps(m)\to_{\minf}$, since we have seen above that the ratio $\disp \dfrac{\int e^{\phi}H_{m}d(\de^{*}\nu_{A'})}{||H_{m|\Om_{D}}||_{\8}}$ is bounded away from zero.

Hence we get, 
\begin{align*}
 \hskip -2.5cm\int r_{DA,m}d\mu_{m}&=(1+\eps(m))e^{-P}e^{-(n_{m}-1)(P-P_{A'})}\left[\int_{\Om_{D}}H_{m}e^{\phi}d\de^{*}\nu_{A'}\right]\sum_{n\ge 0}(n+n_{m})e^{-(n+n_{m})(P_{m}-P)}\left(\left[\int \wt H_{A'}d\nu_{A}\right]\left[\int \wt H_{A}d\nu_{m}\right]\right.\\
 &+\left.  e^{-n\rho}\int\Psi_{A}^{n}(\wt H_{A'}) d\nu_{m}   \right).
\end{align*}

We let the reader check that $\disp \sum_{n\ge 0}(n+n_{m})e^{-(n+n_{m})(P_{m}-P)}$ behaves like 
$\disp \frac{e^{-n_{m}(P_{m}-P)}}{P_{m}-P}\left(n_{m}+\dfrac1{P_{m}-P}\right)$ while $\disp \sum_{n\ge 0}(n+n_{m})e^{-(n+n_{m})(P_{m}-P-\rho)}$ behaves like $e^{-n_{m}(P_{m}-P)}n_{m}$.

Therefore, we finally get (with some rewritings on equivalents)
\begin{equation}
\label{equ1-rdagenegene}
\int r_{DA,m}d\mu_{m}\sim_{\minf}e^{-P}\l_{m}\left(n_{m}+\dfrac1{P_{m}-P}\right)\left[\int_{\Om_{D}}H_{m}e^{\phi}d\de^{*}\nu_{A'}\right]\left[\int \wt H_{A'}d\nu_{A}\right]\left[\int \wt H_{A}d\nu_{m}\right].
\end{equation}

Doing the same work with $r_{AD,m}$ we get 
\begin{equation}
\label{equ1-rationreturngenegene}
\dfrac{\int r_{DA,m}d\mu_{m}}{\int r_{AD,m}d\mu_{m}}\sim_{\minf}\dfrac{\l_{m}}{\l'_{m}}
\dfrac{\nu_{m}(\Om_{A})}{\nu_{m}(\Om_{D})}
\left(\dfrac{n_{m}+\dfrac1{P_{m}-P}}{n'_{m}+\dfrac1{P_{m}-P}}\right)
\dfrac{\left[\int_{\Om_{D}}H_{m}e^{\phi}d\de^{*}\nu_{A'}\right]}{\left[\int_{\Om_{A}}H_{m}e^{\phi}d\al^{*}\nu_{D'}\right]}
\dfrac{\left[\int \wt H_{A'}d\nu_{A}\right]}{\left[\int \wt H_{D'}d\nu_{D}\right]}
\dfrac{\left[\int \wt H_{A}d\nu_{m|A}\right]}{\left[\int \wt H_{D}d\nu_{m|D}\right]}
\end{equation}

Using \eqref{equ3-rapportnum2} and Prop. \ref{prop2-cvnumgenemain1}, \eqref{equ1-rationreturngenegene} yields 

\begin{equation}
\label{equ2-rationreturngenegene}
\dfrac{\int r_{DA,m}d\mu_{m}}{\int r_{AD,m}d\mu_{m}}\sim_{\minf}e^{-P}
\l_{m}
\left(\dfrac{n_{m}+\dfrac1{P_{m}-P}}{n'_{m}+\dfrac1{P_{m}-P}}\right)
\dfrac{\left[\int_{\Om_{D}}H_{m}e^{\phi}d\de^{*}\nu_{A'}\right]}{\left[\int_{\Om_{A}}H_{m}e^{\phi}d\al^{*}\nu_{D'}\right]}
\left[\int \wt H_{A'}d\nu_{A}\right]\left[\int \wt H_{A}e^{\phi}d\al^{*}\nu_{D'}\right].
\end{equation}

\subsubsection{Proof in the case $\l_{m}\to+\infty$}
In that case  $\dfrac1{\l_{m}}H_{m|\Om_{A}}$ goes to $K_{A}\wt H_{A}$ and $H_{m|\Om_{D}}$ goes to $K_{D}\wt H_{D}$, with  (see \ref{eq1-KAKDgenegene}):
$$\dfrac{K_{D}}{K_{A}}=\dfrac{e^{P}}{\int \wt H_{A'}d\nu_{A}\int \wt H_{D}e^{\phi}d\de^{*}\nu_{A'}}$$
Hence 
\begin{equation}
\label{equ3-rationreturngenegene}
\dfrac{\int r_{DA,m}d\mu_{m}}{\int r_{AD,m}d\mu_{m}}\sim_{\minf}
\left(\dfrac{n_{m}+\dfrac1{P_{m}-P}}{n'_{m}+\dfrac1{P_{m}-P}}\right).
\end{equation}

We remind $\disp \l_{m}:=\dfrac{e^{(n_{m}-1)(P_{A'}-P_{m})}}{1- e^{-(P_{m}-P)}}$ which yield 
$$\l_{m}\asymp \dfrac{e^{-n_{m}(P_{m}-P_{A'})}}{P_{m}-P}.$$
In the case $\l_{m}\to+\8$, for $m$ sufficiently large we get $\dfrac1{P_{m}-P}\ge e^{n_{m}(P-P_{A'})}$, hence $\dfrac1{P_{m}-P}$ is much bigger than $n_{m}$. 

Furthermore,  assumption $\dfrac{n_{m}}{n'_{m}}\to\theta\in]0,+\8[$  yields $\dfrac1{P_{m}-P}\ge e^{n'_{m}\theta(P-P_{A'})}$, and then $\dfrac1{P_{m}-P}$ is much bigger than $n'_{m}$.

\begin{remark}
\label{rem-theta}
This is where we use $\theta\in ]0,+\8[$.
$\blacksquare$\end{remark}

Then, \eqref{equ3-rationreturngenegene} shows that $\dfrac{\int r_{DA,m}d\mu_{m}}{\int r_{AD,m}d\mu_{m}}$ goes to 1.  Then,  Proposition \ref{prop-returnmeas} finishes the proof.

\subsubsection{Proof in the case $\l_{m}\nrightarrow0,+\infty$}
Then if $\l_{m}$ goes to $\l$, $H_{m|\Om_{A}}$ goes to $K_{A}\wt H_{A}$ and $H_{m|\Om_{D}}$ goes to $K_{D}\wt H_{D}$, with (see \eqref{eq2-KAKDgenegene})
$$\dfrac{K_{D}}{K_{A}}=\dfrac{e^{P}}{\l\int \wt H_{A'}d\nu_{A}\int \wt H_{D}e^{\phi}d\de^{*}\nu_{A'}}.
$$
 and we still get 
$$\dfrac{\int r_{DA,m}d\mu_{m}}{\int r_{AD,m}d\mu_{m}}\sim_{\minf}
\left(\dfrac{n_{m}+\dfrac1{P_{m}-P}}{n'_{m}+\dfrac1{P_{m}-P}}\right).
$$
Then the arguments are the same: $\dfrac1{P_{m}-P}$ is exponentially bigger than $n_{m}$ and $n'_{m}$, and  
$\dfrac{\int r_{DA,m}d\mu_{m}}{\int r_{AD,m}d\mu_{m}}$ goes to 1.  Again,  Proposition \ref{prop-returnmeas} finishes the proof. 

\section{Proof of Theorem \ref{th-main2-1}} \label{sec-proofmain2-1}
Theorem \ref{th-main2-1} can be obtained from Theorem \ref{th-main1-2} doing $A'=A$ and $D'=D$. In that case, all the computations done in Section \ref{sec-proof-main-1-2}, only the conclusion has to change, since, in that case 
$$P_{A'}=P=P_{D'}.$$ 
Note that in that case, $\nu_{A}=\nu_{A'}$, $\nu_{D'}=\nu_{D}$ and $\disp\int\wt H_{A'}d\nu_{A}=1=\int\wt H_{D'}d\nu_{D}$.

Equation \eqref{equ3-rationreturngenegene} is still valid. 

\subsubsection{The case $\l_{m}\to+\infty$}

In that case $\l'_{m}$ goes to $0$ as $m$ goes to $+\8$. Furthermore, 
$$\l'_{m}\asymp \dfrac{e^{n'_{m}(P-P_{m})}}{P_{m}-P},$$
which shows that $n'_{m}(P_{m}-P)$ goes to $+\8$. 

Note that \eqref{equ3-rationreturngenegene} is equivalent to 
$$\dfrac{\int r_{DA,m}d\mu_{m}}{\int r_{AD,m}d\mu_{m}}\sim_{\minf}\dfrac{n_{m}(P_{m}-P)+1}{n'_{m}(P_{m}-P)}$$

Then two cases occur : 
\begin{enumerate}
\item If $n_{m}(P_{m}-P)$ goes to $+\8$, then $\dfrac{\int r_{DA,m}d\mu_{m}}{\int r_{AD,m}d\mu_{m}}$ behaves as $\dfrac{n_{m}}{n'_{m}}$ and goes to $\theta$. Proposition \ref{prop-returnmeas} finishes the proof. 
\item If $n_{m}(P_{m}-P)$ is bounded from above, then $\dfrac{\int r_{DA,m}d\mu_{m}}{\int r_{AD,m}d\mu_{m}}$ goes to 0. But in that case note that 
$\dfrac{n_{m}}{n'_{m}}=\dfrac{n_{m}(P_{m}-P)}{n'_{m}(P_{m}-P)}$ and then $\theta=0$. Again, Proposition \ref{prop-returnmeas} finishes the proof.
\end{enumerate}

\subsubsection{The case $\l_{m}\nrightarrow 0,+\infty$}
In that case both $n_{m}(P_{m}-P)$ and $n'_{m}(P_{m}-P)$ go to $+\8$. 
Then, $$\dfrac{\int r_{DA,m}d\mu_{m}}{\int r_{AD,m}d\mu_{m}}\sim_{\minf}\dfrac{n_{m}(P_{m}-P)+1}{n'_{m}(P_{m}-P)+1}$$ 
and $\dfrac{n_{m}}{n'_{m}}=\dfrac{n_{m}(P_{m}-P)}{n'_{m}(P_{m}-P)}$ yields 
$$\dfrac{\int r_{DA,m}d\mu_{m}}{\int r_{AD,m}d\mu_{m}}\to_{\minf}\theta.$$

\section{Proof of Theorem \ref{th-main1-3}}
\label{sec-proofmain3}

\subsection{Construction of $\wt T_{\eps}$}

Following \cite{Horan20} we pick some map $T_{\eps_{0}}$ which admits a Markov partition as in Figure \ref{fig-map2}.
\begin{figure}[htbp]
\begin{center}
\includegraphics[scale=0.5]{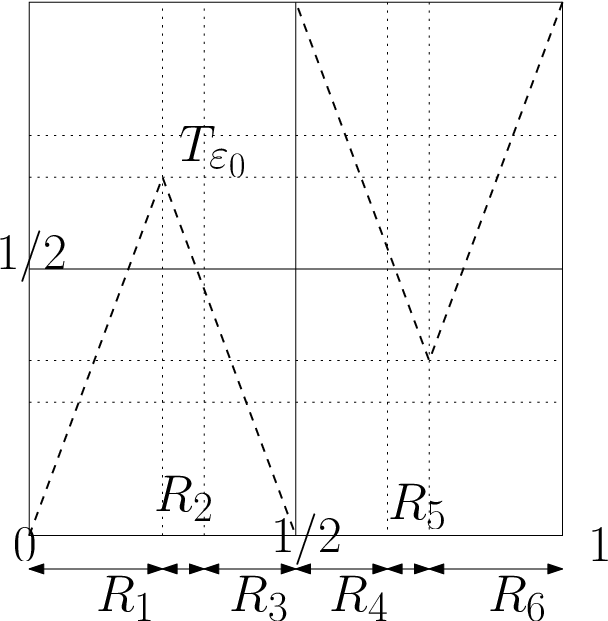}
\caption{{\bf Map $T_{\eps_{0}}$ and Markov partition}}
\label{fig-map2}
\end{center}
\end{figure}

The map $T_{\eps_{0}}$ is piecewise affine with derivative bigger than 2 in absolute value. 

The associated transition matrix is 
$M_{0}:=\left(\begin{array}{cccccc}1 & 1 & 1 & 0 & 0 & 0 \\0 & 0 & 0 & 1 & 0 & 0 \\1 & 1 & 1 & 0 & 0 & 0 \\0 & 0 & 0 & 1 & 1 & 1 \\0 & 0 & 1 & 0 & 0 & 0 \\0 & 0 & 0 & 1 & 1 & 1\end{array}\right)$ and both matrices $A$ and $D$ are equal to 
$\left(\begin{array}{ccc}1 & 1 & 1 \\0 & 0 & 0 \\1 & 1 & 1\end{array}\right)$.

One fixes two increasing sequences  $n_{m}$ and $n'_{m}$ of integers going to $+\8$ and such that $\dfrac{n_{m}}{n'_{m}}$ goes to some $\theta\in ]0,+\8[$ (to get Theorem \ref{th-main1-2}). 
The two pre-images by $T_{\eps_{0}}$ of $\frac12$ by the left-subsystem are respectively denoted by $a$ and $a'$. The preimages for the right-subsystems are respectively denotes by $c$ and $c'$. $b$ is the preimage of $\frac14$ by the first branch of the left-subsystem (see Fig. \ref{fig-maps6}).

Then, for each $m$, one constructs the map $T_{\eps_{m}}$ as follows. The two subsystems (left and right) are constructed from $n_{m}$ and $n'_{m}$, thus the symmetry may be broken. For simplicity we denote it by $T_{\eps_{m}}$ even if it will actually depend on some $\eps'_{m}$ on the left hand side and some $\eps_{m}$ on the right hand side. 

\subsubsection{Construction of the right subsystem}. 

We adjust $\eps_{m}$ such that $T_{\eps_{0}}^{n_{m}}(d)=\frac14$, with $d=\frac12-\eps_{m}$. This is done such that $T_{\eps_{0}}(d)$ stays for $n_{m}-1$ iterates in $R_{1}$. 
Then, we pick a piecewise expanding map on the interval $\left[\frac12,1\right]$, with two branches, as the one in dash-dotted line in Figure \ref{fig-maps6}. 
The peak has value $\frac12-\eps_{m}=d$.

\begin{figure}[htbp]
\begin{center}
\includegraphics[scale=0.7]{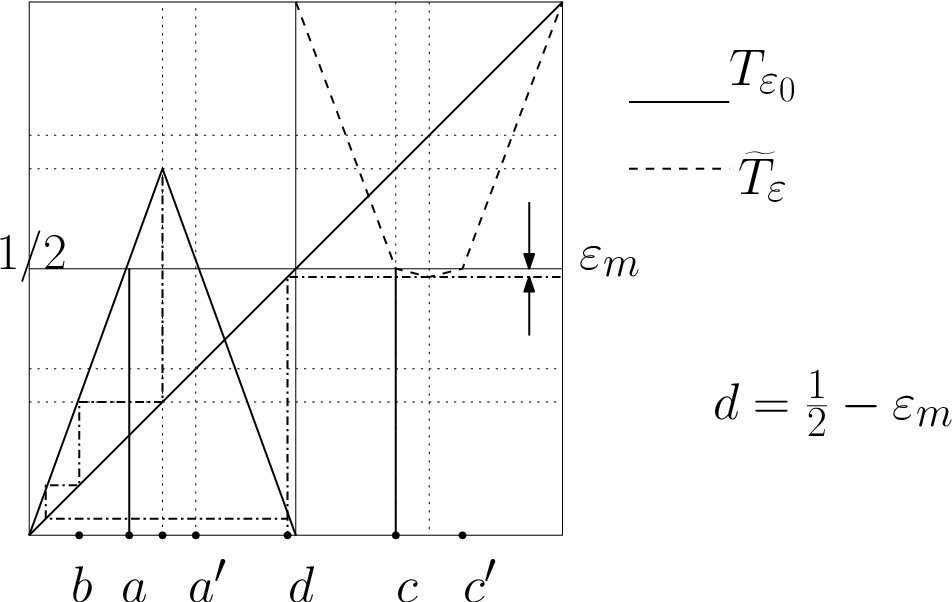}
\caption{\bf construction of the right-hand side of the map $T_{\eps}$ at step $m$}
\label{fig-maps6}
\end{center}
\end{figure}

The map $\wt T_{\eps_{m}}$ (dashed line in Figure \ref{fig-maps6}) coincides with $T_{\eps_{0}}$ on $\left[\frac12,c\right]$ and $[c',1]$. On $\left[c,\frac34\right]$ and $\left[\frac34,c'\right]$ it is monotonous and affine with 
$$\wt T_{\eps_{m}}(c)=\wt T_{\eps_{m}}(c')=\frac12\text{ and }\wt T_{\eps_{m}}\left(\frac34\right)=\frac12-\eps_{m}.$$

\subsubsection{Construction of the left subsystem}

The construction in $\left[0,\frac12\right]$ is done in a similar way, replacing $n_{m}$ by $n'_{m}$, $c$ and $c'$ by $a$ and $a'$, $\eps_{m}$ by $\eps'_{m}$ and $d$ by $\frac12+\eps'_{m}$.  

\subsubsection{The map $\wt T_{\eps_{m}}$}

We set $A'=\left(\begin{array}{ccc}1 & 0 &0  \\0 & 0 & 0 \\1 & 0 &0 \end{array}\right)$ and $D'=\left(\begin{array}{ccc}1 & 0 &0  \\0 & 0 & 0 \\1 & 0 &0 \end{array}\right)$. Then $A'<A$ and $D'<D$ hold. We claim that $\wt T_{\eps_{m}}$ is conjugated\footnote{ as usually, the conjugacy is up to a countable set of points, that comes from the boundaries of the Markov partition.} to $\S_{m}$ (with the matrices $A$, $D$, $M_{0}$, $A'$ and $D'$).

We associate to $R_{i}$ the letter $\al_{i}$ with $i=1,2,3$ and to $R_{i}$ the letter $\de_{i}$ with $i=4,5,6$. 
If $x$ is in $[0,1]$, we can associate to $x$ an infinite word, $\ul x$ within the alphabet $\{\al_{1},\al_{2},\al_{3},\de_{4},\de_{5},\de_{6}\}$. The word is unique for points $x$ whose orbit avoids the boundaries of the $R_{i}$'s.

After $\al_{1}$ and $\al_{3}$ one can get any $\al_{i}$. After $\al_{1}$ one can also get $\de_{4}$ and after $\al_{2}$ one can only get $\de_{4}$. After $\al_{1}\de_{4}$ or $\al_{3}\de_{4}$ one can only see $\de_{1}^{n'_{m}-2}$ and then any $\de_{i}$. 

Similarly, after $\de_{4}$ and $\de_{6}$ one can see any $\de_{i}$. After $\de_{6}$ one can also see $\al_{3}$, and after $\de_{5}$ one can only see $\al_{3}$. In these later cases, one has to see $\al_{1}^{n_{m}-2}$ and then any $\al_{i}$. 

This shows that each $\ul x$ belongs to $\S_{m}$. 

\bigskip
Conversely, we get 
\begin{enumerate}
\item $\wt T_{\eps_{m}}(R_{1})=R_{1}\cup R_{2}\cup R_{3}\cup \left[\frac12,\frac12+\eps'_{m}\right]$. 
\item $\left[\frac12,\frac12+\eps'_{m}\right]$ is included into $R_{4}$, and $\wt T^{n'_{m}}_{\eps_{m}}(\left[\frac12,\frac12+\eps'_{m}\right])=\left[\frac34,1\right]=R_{6}$. 
\item $\wt T_{\eps_{m}}(R_{1})=\left[\frac12,\frac12+\eps'_{m}\right]$. 
\item $\wt T_{\eps_{m}}(R_{1})=R_{1}\cup R_{2}\cup R_{3}$.
\item $\wt T_{\eps_{m}}(R_{4})=R_{4}\cup R_{5}\cup R_{6}$.
\item $\wt T_{\eps_{m}}(R_{5})=\left[\frac12-\eps_{m}\frac12\right]$.
\item $\wt T_{\eps_{m}}(R_{1})=R_{4}\cup R_{5}\cup R_{6}\cup \left[\frac12-\eps_{m}\frac12\right]$.
\item $\left[\frac12-\eps_{m},\frac12\right]$ is included into $R_{3}$ and $\wt T^{n_{m}}_{\eps_{m}}(\left[\frac12-\eps_{m},\frac12\right])=R_{1}$. 
\end{enumerate}
This yields, that if $\ul x$ is an infinite word eligible for $\S_{m}$, then, for every $n$, the cylinder $[x_{0}\ldots x_{n}]$ corresponds in $[0,1]$ to a non empty segment $[a_{n},b_{n}]$. Furthermore, the sequence of the segments is decreasing and thus the infinite intersection contains at least one point.

\medskip
It remains to check that two different points in $[0,1]$ must have two different codes in $\S_{m}$. For that, it is sufficient to check that any point has a positive upper Lyapunov exponent. 

For that, we pick a point $x$, and we assume it belongs to $[a,a']$. For simplicity we assume it belongs to $\left[a,\frac14\right]$. 
Then $\wt T'_{\eps'_{m}}(x)=\dfrac{\eps_{m}}{\frac14-a}$. 
For the next $n'_{m}$ iterates, the orbit belongs to $R_{4}$ where there is constant expansion, say $\l$. 
Furthermore, we have 
$$\l^{n'_{m}}\eps' _{m}=\frac14,$$
since $\wt T^{n'_{m}}_{\eps_{m}}(\frac12+\eps'_{m})=\frac34$ and $T_{\eps_{m}}$ is affine. Hence we get 
$$\left(\wt T^{n'_{m}+1}\right)'(x)=\l^{n_{m}}\frac{\eps_{m}}{\frac14-a}=\dfrac1{1-4a}>1.$$
Then, the orbit may stay where $\wt T_{\eps_{m}}$ is uniformly expanding for some iterates (corresponding to the $n$ for $\de^{n}$ in the computations  in Section \ref{sec-proof-main-1-2}), and finally fall into $[c,c']$. Then the computation is similar to the one we have done. 

The reasoning is the same for $x\in \left[\frac14,a'\right]$. 

This yields that if we decompose the orbit in 
$$n'_{m}+1, N_{1},n_{m}+1,N_{2},\ldots$$
the derivative after $k$ loops is at least 
$$\ga^{k}\l^{N_{1}+\ldots N_{k}},$$
with $\ga=\min\left(\dfrac1{1-4a},\dfrac1{4a'-1},\dfrac1{3-4c},\dfrac1{4c'-3}\right)$. 
The length of the orbit is $\sim \frac{k}2(n'_{m}+n_{m}+2)+N_{1}+\ldots N_{k}$. 
Then, the mean value for the log of the derivative is greater than 
$$\dfrac{k\log\ga+\sum_{i=1}^{k}N_{i}\log \l}{\frac{k}2(n'_{m}+n_{m}+2)+\sum_{i=1}^{k}N_{i}}>\dfrac1{n'_{m}+n_{m}+2}\min(\log\ga,\log\l)=:\log\ga'>0.$$

This shows that the sequence of intervals $[a_{n},b_{n}]$ is decreasing and for infinitely many $n$ we get 
$$0\le b_{n}-a_{n}\le (\ga')^{-n},$$
which goes to 0 as $n$ goes to $+\8$. Therefore, $\wt T_{\eps_{m}}$ is conjugated to $\S_{m}$. 

\subsection{The map $T_{\eps}$}

\subsubsection{Principle of the construction}

We will construct a map say $T_{\eps}$ which is conjugated to $\wt T_{\eps_{m}}$ and with shape like a double tent-maps with small hole (see Fig.  \ref{fig-Teps}).  It will be almost everywhere expanding. 

\begin{figure}[htbp]
\begin{center}
\includegraphics[scale=0.5]{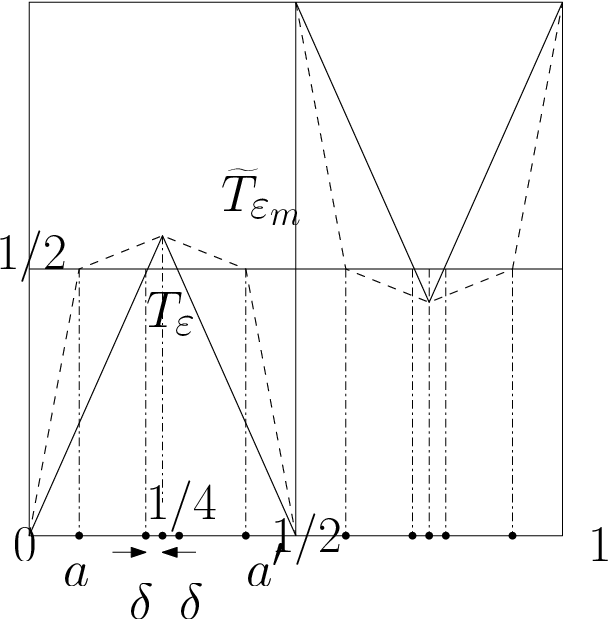}
\includegraphics[scale=0.5]{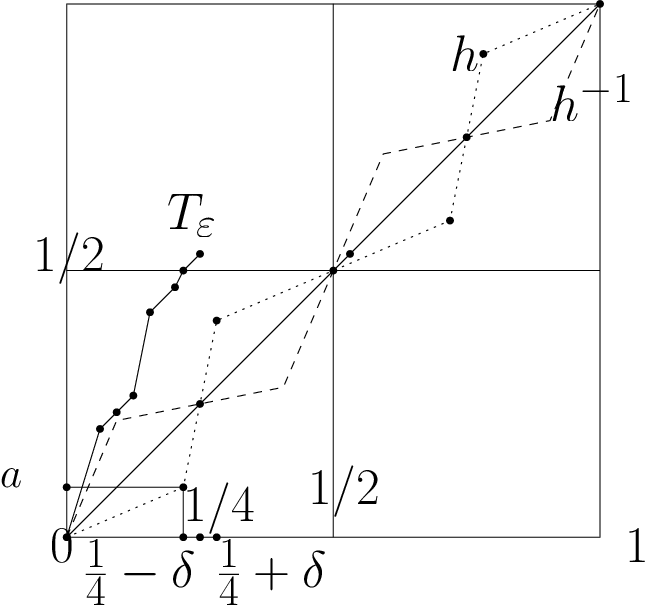}
\caption{{\bf Construction of $T_{\eps}$. Conjugacies $h$, $h^{-1}$ and first branch of $T_{\eps}$}}
\label{fig-Teps}
\end{center}
\end{figure}

The construction of the conjugacy map $h$ is done by induction. The map $\wt T_{\eps}$ is piecewise affine on the intervals $[0,a]$, $[a,\frac14]$, $[\frac14, a']$ and $[a',\frac12]$ (for the first half). It has symmetry in the sense that $a'=\frac12-a$.

The construction is done in a similar way on the second half of the interval. 

We pick  some positive $\delta$  very close to 0. 

The map $h$ maps $[\frac14-\delta,\frac14]$ in an affine way on $[a,\frac14]$, and the interval $[\frac14,\frac14+\delta]$ in an affine way on $[\frac14,a']$. It has thus a slope proportional to $\dfrac1\delta$. Therefore, $h^{-1}$ has slope proportional to $\delta$ and on the interval $[a,a']$. As we shall make $\delta$ decrease to 0, we need to balance this very small slope in view to get a global map $T_{\eps}=h^{-1}\circ \wt T_{\eps}\circ h$ that is expanding. 

This means that  $h$ must have a big slope on the interval  $h^{-1}\wt T^{-1}_{\eps}([a,a'])$. Again, this big slope for $h$ generates a small slope for $h^{-1}$, that has to be balanced by a big slope in $h^{-1}\wt T^{-2}_{\eps}([a,a'])$, and so on. 

More precisely, we consider the map with 2 branches on $[0,a]$ and $[a',\frac12]$ as in Figure \ref{fig-cecilia10}
\begin{figure}[htbp]
\begin{center}
\includegraphics[scale=0.5]{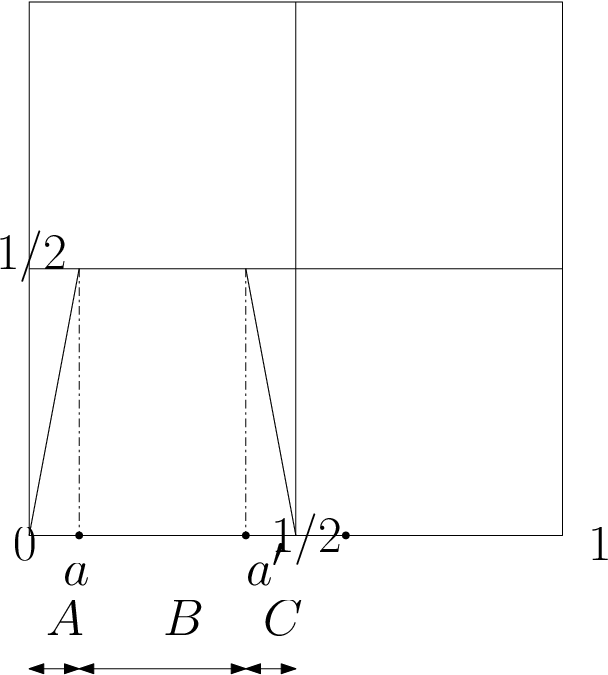}
\caption{{\bf  Cantor set and interval to compute $h$}}
\label{fig-cecilia10}
\end{center}
\end{figure}
Then we consider all the pre-images of the interval $B$ by the map. They are all intervals that we can write as $AA\ldots CAB$, that is a finite word in $A$ and $C$ and then the letter $B$. 
Such an interval has length $(\frac12-2a)(2a)^{n}$ if the word in $A$ and $C$ in front of the final (and unique) $B$ has length $n$. 

The union of these intervals has full Lebesgue measure in $[0,a]\cup [a',\frac12]$ since Lebesgue almost every point will eventually enter into the interval $B$. 

Furthermore, note that $\wt T_{\eps}$ maps an interval $D_{1}D_{2}\ldots D_{n}B$ onto the interval $D_{2}\ldots D_{n}B$ with $D_{i}=A,C$. 

Then, the main idea to construct $h$  is the following:  we shall pick a fat triadic Cantor set (that is with positive Lebesgue measure) and identify the intervals that appear in its construction with the intervals $D_{1}D_{2}\ldots D_{n}B$ with $D_{i}=A,C$. 

\subsubsection{The fat Cantor set and construction of the conjugacy}

We pick $\be<\frac13$ and construct in $[0,1]$ the fat Cantor set as follows. The exact magnitude of $\be$ will be make precise later. 
\begin{enumerate}
\item Step 1, we remove one interval, say $C_{1}$ of length $\be$ centered in the middle. There are two surviving interval called $E_{1}$'s. 
\item Step 2, we remove in each surviving interval $E_{1}$ an interval, $C_{2}$ of length $\beta^{2}$, centered in the middles. There are 4 surviving intervals, $E_{2}$'s. 
\item Step $n$, we remove in each surviving interval $E_{n-1}$ a centered interval $C_{n}$ of length $\beta^{n}$. 
\end{enumerate}
The surviving Cantor set has positive Lebesgue since the union of removed intervals has total length 
$$\be+2\be^{2}+4\beta^{3}+\ldots +2^{n-1}\be^{n}+\ldots=\dfrac{\be}{1-2\be}<1,$$
since $\be<\frac13$.

We rescale the total length by the factor $\frac14-\delta$ and for simplicity keep the names of the intervals. 
Note that by construction, each $C_{n}$ is included into one $E_{n-1}$, and each $E_{n}$ is between\footnote{except the first and the last ones. } two $C_{j}$ and $C_{j'}$, with $j,j'\le n$.

 We also denote by $B_{n}$ and interval in $[0,a]$ (or $[a',\frac12]$) of the form $D_{1}\ldots D_{n}B$ with $D_{i}=A,C$. 
There is a canonical order on these intervals given by the usual inequality $\le$ in $\R$. 

Each $E_{n}$ has length $|E_{n}|=(\frac14-\delta)\dfrac{1-\be\frac{1-(2\be)^{n}}{1-2\be}}{2^{n}}$, and each 
$C_{n}$ has length $|C_{n}|=(\frac14-\delta)\be^{n}$. 
Furthermore, $B_{n}$ has length $|B_{n}|=(\frac12-2a)(2a)^{n}$.

\medskip
The map $h_{n}$ is obtained by induction in the following way:
\begin{enumerate}
\item Step 0, $h_{0}$ is the increasing affine map that maps $[0,\frac14-\delta]$ on $[0,a]$ and $[\frac14+\delta,\frac12]$ on $[a',\frac12]$. 
\item at step $n$, $h_{n}$  maps in an affine way any interval $C_{n}$ onto the corresponding (for the order relation) interval $B_{n}$. It links also in an affine way right end points to the left end point of the next interval.
\end{enumerate}
Figure \ref{fig-conjugacyh} shows how to get $h_{n+1}$ from $h_{n}$

\begin{figure}[htbp]
\begin{center}
\includegraphics[scale=0.5]{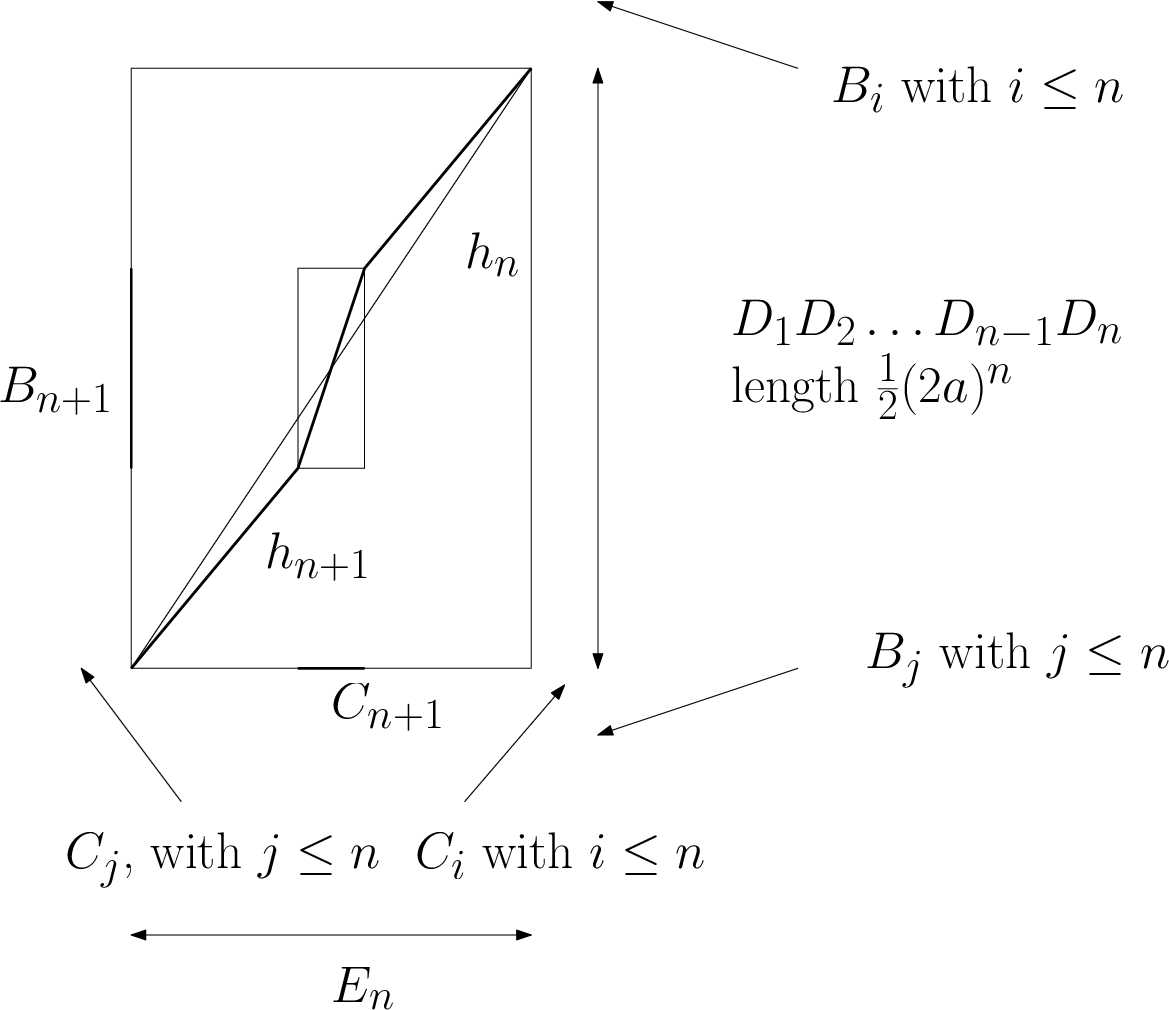}
\caption{{\bf Construction of $h$ by induction}}
\label{fig-conjugacyh}
\end{center}
\end{figure}

The difference between $h_{n}$ and $h_{n+1}$ is bounded by $|B_{n+1}|+|C_{n+1}|$ thus the sequence of $h_{n}$ is a Cauchy sequence of continuous maps for the $||\ ||_{\8}$-norm. 
This shows that is converges to an increasing maps from $[0,\frac14-\delta]$ onto $[0,a]$ (and also from $[\frac14+\delta,\frac12]$ onto $[a',\frac12]$). 

On each $C_{n}$, $h$ has constant slope $\dfrac{(\frac12-2a)(2a)^{n-1}}{(\frac14-\delta)\be^{n}}$. On each $E_{n}$, $h_{n}$ has slope 
$$\dfrac{\frac12(2a)^{n+1}}{(\frac14-\delta)(\frac{1-\be\frac{1-(2\be)^{n}}{1-2\be}}{2^{n+1}}-\be^{n+1})}.$$ 
As we said above we have the chain 
$$C_{n}\stackrel{h}{\longrightarrow}B_{n}\stackrel{\wt T_{\eps}}{\longrightarrow}B_{n-1}\stackrel{h^{-1}}{\longrightarrow}C_{n-1},$$
then, the global slope for $n\ge 2$ is 
$$\dfrac{(\frac12-2a)(2a)^{n-1}}{(\frac14-\delta)\be^{n}}\frac1{2a}\dfrac{(\frac14-\delta)\be^{n-1}}{(\frac12-2a)(2a)^{n-2}}=\frac1\be>1.$$
The global slope on $C_{1}$ is 
\begin{equation}\label{eq-condibeta}
\dfrac{(\frac12-2a)}{(\frac14-\delta)\be}\frac1{2a}\frac\delta{\frac14-a}=\frac{\de}{(\frac14-a)\be a}.
\end{equation}
For points that are in no $C_{n}$, the slope for $h_{n}^{-1}\circ \wt T_{\eps}\circ h_{n}$ is approximatively 2. 

\subsubsection{For points in the cusp}
points in the interval $[\frac14-\de,\frac14+\de]$ are send by $h$ in $[a,a']$ which is sent by $\wt T_{\eps}$ in the second half of the whole interval (namely $[\frac12,1]$). However we remind that the top point is sent, within our notations, to a middle point of some $B_{n}$ (but in the other half of the interval). Because $h$ is increasing, when applying $h^{-1}$ we get an interval $[\frac12,b_{n}]$, with $b_{n}$ the middle point of the $C_{n}$  associated to the previous $B_{n}$. The same holds for second half of the interval (with the cusp in the first half). 

Then, the construction matches.

\bibliographystyle{plain}
\bibliography{bib-overleaf}

\begin{thebibliography}{10}

\bibitem{bahsoun-saussol}
Wael Bahsoun and Beno\^it Saussol.
\newblock Linear response in the intermittent family: differentiation in a
  weighted {$C^0$}-norm.
\newblock {\em Discrete Contin. Dyn. Syst.}, 36(12):6657--6668, 2016.

\bibitem{bahsoun-vaienti}
Wael Bahsoun and Sandro Vaienti.
\newblock Metastability of certain intermittent maps.
\newblock {\em Nonlinearity}, 25(1):107--124, 2012.

\bibitem{lelievre1}
Andrew Binder, Tony Leli\`evre, and Gideon Simpson.
\newblock A generalized parallel replica dynamics.
\newblock {\em J. Comput. Phys.}, 284:595--616, 2015.

\bibitem{Bovier-DenHollander}
Anton Bovier, Frank den Hollander, and Saeda Marello.
\newblock Metastability for {G}lauber dynamics on the complete graph with
  coupling disorder.
\newblock {\em Comm. Math. Phys.}, 392(1):307--345, 2022.

\bibitem{buzzi-course}
J\'{e}r\^{o}me Buzzi.
\newblock Chaos and ergodic theory.
\newblock In {\em Mathematics of complexity and dynamical systems. {V}ols.
  1--3}, pages 63--87. Springer, New York, 2012.

\bibitem{BKL}
J\'er\^ome Buzzi, Beno\^it Kloeckner, and Renaud Leplaideur.
\newblock Nonlinear thermodynamical formalism.
\newblock {\em Ann. H. Lebesgue}, 6:1429--1477, 2023.

\bibitem{lelievre2}
Giacomo Di~Ges\`u, Tony Leli\`evre, Dorian Le~Peutrec, and Boris Nectoux.
\newblock The exit from a metastable state: concentration of the exit point
  distribution on the low energy saddle points, part 1.
\newblock {\em J. Math. Pures Appl. (9)}, 138:242--306, 2020.

\bibitem{dolgopyat}
Dmitry Dolgopyat.
\newblock On differentiability of {SRB} states for partially hyperbolic
  systems.
\newblock {\em Invent. Math.}, 155(2):389--449, 2004.

\bibitem{dolgopyat-wright}
Dmitry Dolgopyat and Paul Wright.
\newblock The diffusion coefficient for piecewise expanding maps of the
  interval with metastable states.
\newblock {\em Stoch. Dyn.}, 12(1):1150005, 13, 2012.

\bibitem{dubbeldam}
Johan L.~A. Dubbeldam, Vicente~Lenz Burnier, Elena Pulvirenti, and Martin
  Slowik.
\newblock Metastability for the curie-weiss-potts model with unbounded random
  interactions, 2025.

\bibitem{tokman1}
Cecilia Gonz\'alez-Tokman, Brian~R. Hunt, and Paul Wright.
\newblock Approximating invariant densities of metastable systems.
\newblock {\em Ergodic Theory Dynam. Systems}, 31(5):1345--1361, 2011.

\bibitem{Horan20}
Joseph Horan.
\newblock Dynamical spectrum via determinant-free linear algebra.
\newblock January 2020.
\newblock arXiv:2001.06788 [math].

\bibitem{keller-liverani}
Gerhard Keller and Carlangelo Liverani.
\newblock Stability of the spectrum for transfer operators.
\newblock {\em Ann. Scuola Norm. Sup. Pisa Cl. Sci. (4)}, 28(1):141--152, 1999.

\bibitem{kloeckner}
Beno\^it~R. Kloeckner.
\newblock The linear request problem.
\newblock {\em Proc. Amer. Math. Soc.}, 146(7):2953--2962, 2018.

\bibitem{leplaideur-flot}
Renaud Leplaideur.
\newblock Totally dissipative measures for the shift and conformal
  {$\sigma$}-finite measures for the stable holonomies.
\newblock {\em Bull. Braz. Math. Soc. (N.S.)}, 41(1):1--36, 2010.

\bibitem{leplaideur-synthese}
Renaud Leplaideur.
\newblock From local to global equilibrium states: thermodynamic formalism via
  an inducing scheme.
\newblock {\em Electron. Res. Announc. Math. Sci.}, 21:72--79, 2014.

\bibitem{leplaideur-watbled1}
Renaud Leplaideur and Fr\'ed\'erique Watbled.
\newblock Generalized {C}urie-{W}eiss model and quadratic pressure in ergodic
  theory.
\newblock {\em Bull. Soc. Math. France}, 147(2):197--219, 2019.

\bibitem{leplaideur-watbled2}
Renaud Leplaideur and Fr\'ed\'erique Watbled.
\newblock Curie-{W}eiss type models for general spin spaces and quadratic
  pressure in ergodic theory.
\newblock {\em J. Stat. Phys.}, 181(1):263--292, 2020.

\bibitem{Mahieu-Picco}
P.~Mathieu and P.~Picco.
\newblock Metastability and convergence to equilibrium for the random field
  {C}urie-{W}eiss model.
\newblock {\em J. Statist. Phys.}, 91(3-4):679--732, 1998.

\bibitem{Penrose-lebowitz}
O.~Penrose and J.~L. Lebowitz.
\newblock Rigorous treatment of metastable states in the van der
  {W}aals-{M}axwell theory.
\newblock {\em J. Statist. Phys.}, 3:211--236, 1971.

\bibitem{Ruelle-rig}
David Ruelle.
\newblock {\em Statistical mechanics: {R}igorous results}.
\newblock W. A. Benjamin, Inc., New York-Amsterdam, 1969.

\bibitem{ruelle-diff}
David Ruelle.
\newblock Differentiation of {SRB} states.
\newblock {\em Comm. Math. Phys.}, 187(1):227--241, 1997.

\end{thebibliography}

\end{document}